%% file: main.tex
\documentclass[smallextended,referee,envcountsect]{svjour3}
\smartqed
\usepackage{graphicx}
\usepackage{booktabs}
\usepackage{etoolbox}
\AtBeginEnvironment{table}{\setlength{\extrarowheight}{2pt}}
\AtEndEnvironment{table}{\setlength{\extrarowheight}{1pt}}
\usepackage{longtable}
\usepackage{xcolor}
\usepackage{colortbl}
\usepackage{arydshln}
\usepackage[load-configurations=version-1]{siunitx}
\usepackage{mathtools}
\usepackage{bm}
\usepackage{subfigure}
\usepackage[numbers]{natbib}
\usepackage[colorlinks=true,
            linkcolor=blue!60!black,
            citecolor=blue!60!black,
            urlcolor=blue!60!black,
            bookmarks=true]{hyperref}

\input{math_commands.tex}

\newcommand{\norm}[1]{\lVert #1\rVert}
\DeclarePairedDelimiter{\normf}{\|}{\|_\mathrm{F}}
\DeclarePairedDelimiter{\norms}{\|}{\|_{\mathrm{2}}}
\DeclarePairedDelimiter{\normc}{\|}{\|_{\mathrm{C}}}
\DeclarePairedDelimiter{\normn}{\|}{\|_{*}}
\DeclarePairedDelimiter{\normkfk}{\|}{\|_{\text{KF-}k}}
\DeclarePairedDelimiter{\normfstar}{\|}{\|_\mathrm{F*}}
\DeclarePairedDelimiter{\normftwo}{\|}{\|_\mathrm{F2}}
\DeclarePairedDelimiter{\abs}{\lvert}{\rvert}

\def\<#1,#2>{\langle #1,#2\rangle}

\DeclareMathOperator{\tr}{tr}
\DeclareMathOperator{\diag}{diag}

\renewcommand{\epsilon}{\varepsilon}
\newcommand{\Rmn}{\R^{m\times n}}

\newcommand{\figureref}[1]{Fig.~\ref{#1}}
\newcommand{\tableref}[1]{Table~\ref{#1}}
\newcommand{\sectionref}[1]{Sect.~\ref{#1}}

\newcommand{\lemmaref}[1]{Lemma~\ref{#1}}
\newcommand{\corollaryref}[1]{Corollary~\ref{#1}}

\journalname{JOTA}

\begin{document}

\title{Ky Fan Norms and Beyond: Dual Norms and Combinations for Matrix Optimization}
\titlerunning{Ky Fan Norms and Beyond for Matrix Optimization}

\author{Alexey Kravatskiy \and Ivan Kozyrev \and Nikolai Kozlov \and Alexander Vinogradov \and Daniil Merkulov \and Ivan Oseledets}

\institute{Alexey Kravatskiy, Corresponding author \at
             MIRIAI \\
             Russia \\
             \email{kravatskii.a@miriai.org}
           \and
             Ivan Kozyrev \at
             MIPT, INM RAS \\
             Russia \\
             \email{kozyrev.in@phystech.edu}
           \and
             Nikolai Kozlov \at
             MIPT \\
             Russia \\
             \email{kozlov.na@phystech.edu}
           \and
             Alexander Vinogradov \at
             MIPT \\
             Russia \\
             \email{vinogradov.am@phystech.edu}
           \and
             Daniil Merkulov \at
             MIPT, Skoltech, HSE, AI4Science \\
             Russia \\
             \email{daniil.merkulov@phystech.edu}
           \and
             Ivan Oseledets \at
             AIRI, Skoltech, INM RAS \\
             Russia \\
             \email{i.oseledets@skoltech.ru}
}

\date{Received: date / Accepted: date}

\maketitle

\begin{abstract}
   In this article, we explore the use of various matrix norms for optimizing functions of weight matrices, a crucial problem in deep learning. Moving beyond the spectral norm that underlies the Muon update, we leverage the duals of the Ky Fan norms to introduce the \emph{Fanion} family of linear minimization oracle (LMO) algorithms, which are closely related to Muon, $\nu$-SAM, and Dion. Staying inside the LMO, we construct the families of \emph{F-Fanions} and \emph{S-Fanions}, whose updates are convex combinations of the updates of Fanions and Normalized SGD or SignSGD, respectively. The most promising algorithms in these families are \emph{F-Muon} and \emph{S-Muon}. By conducting an extensive empirical study of all three algorithm families across a wide range of tasks and settings, we demonstrate that F-Muon and S-Muon consistently match Muon's performance, while outperforming Muon on a synthetic smooth convex problem.
 \end{abstract}

 \section{Introduction}
 
 Minimizing loss functions in unprecedentedly high-dimensional spaces has recently become an integral and crucial part of training large language models. Hence, new scalable, time- and memory-efficient algorithms have been demanded. Besides the well-known Adam \citep{kingma2014adam} and AdamW \citep{loshchilov2017decoupled}, recently proposed Muon \citep{jordan2024muon} has shown promising results on training very large models \citep{liu2025muonscalablellmtraining}. Its key difference from Adam and AdamW is that it has been constructed specifically for optimizing functions of weight matrices, which are common in deep learning.
 
 From a theoretical perspective, a key innovation of Muon was its principled derivation of the update rule, which emerged as the solution to an optimization problem constrained by the RMS-to-RMS norm (a scaled version of the spectral norm) \citep{bernstein2025deriving}.
 
 Motivated by the success of Muon, many generalizations and variations of it were proposed. Among the notable ones are Scion \citep{pethick2025training}, Dion \citep{ahn2025dion} and Gluon \citep{riabinin2025gluon}. Those works try to explain Muon's efficiency and establish convergence bounds. One central question, however, remains unanswered:
 
 \emph{Can the employment of a norm other than an operator norm in Muon-like algorithms yield comparable or superior empirical results?}
 
 In this article, we tackle this question by demonstrating that there are many viable non-operator norms. We leverage the family of norms dual to Ky Fan $k$-norms to derive a new family of \emph{Fanions}, algorithms with low-rank updates. This approach theoretically explains the backbone of Dion's update \citep{ahn2025dion} and generalizes the memory-motivated application of the nuclear norm to Sharpness-Aware Minimization \citep{pethick2025sam}. As was done with Muon, we develop an effective procedure for computing Fanions' updates: the Lanczos algorithm is the solution (see \sectionref{sec:matrix-side}).
 
 Working with dual norms and various convex combinations of norms, we construct the families of \emph{F-Fanions} and \emph{S-Fanions}, which are hybrids of Fanions with Normalized SGD and SignSGD, respectively.
 
 In \sectionref{sec:experiments}, we compare the performance of these algorithm families on various model and real-world problems:
 \begin{itemize}
   \item Synthetic smooth convex problem with a matrix argument
   \item CIFAR-10 airbench \citep{cifar2023airbench}
   \item Pre-training NanoGPT and GPT-2 Medium on the FineWeb dataset \citep{modded_nanogpt_2024}
   \item Fine-tuning NanoGPT on TinyStories dataset \citep{eldan2023tinystoriessmalllanguagemodels}
 \end{itemize}
 
 Our experiments reveal important insights into the role of matrix norms in optimization. First, using the example of Neon (the rank-one Fanion), we show that not every LMO-based algorithm is effective, despite sharing the same asymptotics in the general bounds of \citep{kovalev2025understanding} and \citep{riabinin2025gluon}. This suggests that existing theoretical guarantees should be refined to better explain empirical performance.
 
 Second, our experiments on real-world tasks demonstrate that the choice of underlying matrix norm is remarkably flexible. On CIFAR-10 airbench, properly-tuned F-Muon and S-Muon achieve $94.02 \pm 0.13\%$ and $94.03 \pm 0.13\%$ accuracy, matching Muon's $94.01 \pm 0.13\%$ performance. After 1750 iterations on NanoGPT training, F-Muon achieves 3.281 cross-entropy loss, while Muon achieves 3.279. Similarly, after 5960 iterations of GPT-2 Medium training, we get 2.9215 with F-Muon and 2.9198 with Muon. Finally, S-Muon matches Muon in fine-tuning NanoGPT on TinyStories, while F-Muon is far more robust to learning rate choice than Muon. These results show that Muon-like algorithms can maintain competitive performance even when the underlying norm constraint is significantly modified, providing an affirmative answer to the central question posed above. Moreover, the tools from \sectionref{sec:conic_combo} give researchers considerable flexibility in designing algorithms that need not be direct modifications of Muon.
 
 \section{Preliminaries: Linear Minimization Oracle Framework}
 
 Training a neural network is essentially an optimization of a function of several weight matrices and a few vectors. Let us start by considering the problem of minimizing a differentiable function of a \emph{single} matrix (the connection to the general case is presented in \sectionref{sec:lmo_for_nn}):
 \begin{equation}\label{eq:matrix_problem}
   F(\cdot)\colon \Rmn \to \R\,,\qquad F(\mX) \to \min_{\mX \in \Rmn}\,
 \end{equation}
 We equip the matrix space $\Rmn$ with a standard dot product $\<\cdot, \cdot> \to \R$ and a norm $\norm{\cdot}\colon \Rmn \to \R_+$, which does not have to coincide with the Frobenius norm $\normf{\cdot} = \sqrt{\<\cdot,\cdot>}$. The dual norm $\norm{\cdot}^\dagger\colon \Rmn \to \R_+$ that is associated with $\norm{\cdot}$ is defined as
 \begin{equation}\label{eq:dual_norm}
   \norm{\mM}^\dagger = \sup_{\mD \in \Rmn\,:\norm{\mD}\leq 1} \<\mM,\mD>\,.
 \end{equation}
 
 Such problems can be solved with an iterative algorithm based on the Linear Minimization Oracle (LMO):
 \begin{equation}\label{eq:lmo}
   \mathrm{LMO}(\mM^t) \in \argmin_{\mD \in \cS} \langle\mM^t, \mD\rangle\,,
 \end{equation}
 where $\mM^t$ is the effective update direction and $\cS \subset \Rmn$ is a constraint set. The algorithm proceeds as follows:
 \begin{equation}\label{eq:simple_update}
   \begin{aligned}
   \mB^t &= \beta \mB^{t-1} + (1 - \beta) \nabla F(\mX^t) \quad \text{(momentum buffer)},\\
   \mM^t &= \begin{cases}
   \nabla F(\mX^t) & \text{(no momentum)},\\
   \mB^t & \text{(heavy ball)},\\
   \nabla F(\mX^t) + \beta \mB^t & \text{(approximate Nesterov)},
   \end{cases}\\
   \mX^{t + 1} &= \mX^t + \gamma_t\mathrm{LMO}\left(\mM^t\right)\,,
   \end{aligned}
 \end{equation}
 where $\beta \in [0, 1)$ is the momentum coefficient and $\nabla F(\mX^t)$ is either full or stochastic gradient. Throughout our experiments, we employ approximate Nesterov momentum, which matches the implementation in Muon and PyTorch SGD (\texttt{nesterov=True}). It uses the current gradient rather than the true lookahead gradient, trading theoretical guarantees for computational efficiency. From here on, we refer to the stochastic gradient as just the gradient if not otherwise specified.
 
 We are particularly interested in the case when $\cS$ is a unit ball in some norm $\norm{\cdot}$:
 \[
 \cS = \cB_{\norm{\cdot}} = \{ \mD \in \Rmn : \norm{\mD} \leq 1 \}\,.
 \]
 In this case,
 \[
 \argmin_{\mD \in \cS} \<\mM^t, \mD> = - \{ \mD \in \cB_{\norm{\cdot}} : \<\mM^t, \mD> = \norm{\mM^t}^\dagger\}\,,
 \]
 and the update for $\mX^{t + 1}$ in (\ref{eq:simple_update}) simplifies to
 \begin{equation}\label{eq:our_update}
   \mX^{t+1} = \mX^{t} - \gamma_t \{\mD \in \cB_{\norm{\cdot}} : \<\mM^t, \mD> = \norm{\mM^t}^\dagger\}\,.
 \end{equation}
 
 Using this formula, it is easy to compute updates for algorithms induced by various norms $\norm{\cdot}$ \citep{bernstein2024oldoptimizernewnorm, pethick2025training}:
 
 \paragraph{Frobenius norm and Normalized SGD}
 When the norm $\norm{\cdot}$ is the Frobenius norm $\normf{\cdot}$, (\ref{eq:our_update}) turns into
 \begin{equation}\label{eq:nsgd_update}
   \mX^{t+1} = \mX^{t} - \gamma_t \frac{\mM^t}{\normf{\mM^t}}\,,
 \end{equation}
 which recovers Normalized SGD (NSGD).
 
 \paragraph{Spectral norm and Muon}
 When the norm is the spectral norm $\norms{\cdot}$, we get
 \begin{equation}\label{eq:muon_update}
   \mX^{t+1} = \mX^{t} - \gamma_t \mU \mV^\top\,,
 \end{equation}
 which is Muon without the $\sqrt{m/n}$ factor.
 Here, $\mM^t = \mU \mSigma \mV^\top$ is the Singular Value Decomposition (SVD) of $\mM^t$ ($\mU = [\vu_1, \vu_2, \dots, \vu_r]$, $\mSigma = \diag(\sigma_1, \sigma_2, \dots, \sigma_r)$, and $\mV = [\vv_1, \vv_2, \dots, \vv_r]$). Muon can be recovered by taking the RMS-to-RMS norm: $\sqrt{n/m}\norms{\cdot}$.
 
 \paragraph{Chebyshev norm and SignSGD}
 When the norm is the Chebyshev norm $\normc{\cdot}$ (also known as the infinity norm, $\normc{A} \equiv \max_{i,j}\abs{A_{ij}}$), we get
 \begin{equation}\label{eq:signsgd_update}
   \mX^{t+1} = \mX^{t} - \gamma_t \sign(\mM^t)\,,
 \end{equation}
 which recovers SignSGD~\citep{bernstein2018signsgd}. Here, $\sign(\mM^t)$ denotes the element-wise sign function. SignSGD is particularly notable for its communication efficiency in distributed training, as it compresses gradients to 1 bit per parameter.
 
 \section{Beyond the Spectral Norm: Fanions}
 
 \subsection{The Rank Gap Problem}
 
 Having examined the Frobenius norm (NSGD), spectral norm (Muon), and Chebyshev norm (SignSGD), all of which produce full-rank updates, we turn to the nuclear norm $\normn{\cdot}$.
 
 \begin{lemma}\label{lemma:neon_update}
   When $\norm{\cdot} = \normn{\cdot}$, (\ref{eq:our_update}) becomes
   \begin{equation}\label{eq:neon_update}
     \mX^{t+1} = \mX^{t} - \gamma_t \vu_1 \vv_1^\top\,.
   \end{equation}
 \end{lemma}
 \begin{proof}
 Since $\normn{\cdot}^\dagger = \norms{\cdot}$, the goal is to reach $\<\mM^t, \mD> = \sigma_1$ in (\ref{eq:our_update}). Note that $\mD = \vu_1 \vv_1^\top$ delivers this value. Indeed, $\normn{\mD}=1$ and by the trace property and orthogonality of the singular vectors,
 \[
 \<\mM^t, \mD> =\<\mU \mSigma \mV^\top, \vu_1 \vv_1^\top> = \tr \diag({\sigma_1, 0, \dots, 0}) = \norms{\mM^t}\,,
 \]
 which completes the proof.\qed
\end{proof}
 
 We call this algorithm \emph{Neon}. The nuclear norm yields rank-one updates, in stark contrast to the full-rank updates of Muon, NSGD, and SignSGD. This raises a natural question: can we derive algorithms with updates of intermediate ranks?
 
 \paragraph{Schatten norms}
 Cesista~\cite{cesista2025schattenp} considered Schatten-$p$ norms:
 \[\norm{\mM^t}_{S_p} = \left(\sum_{i=1}^{\min(m, n)}\sigma_i^p\right)^{1/p}\,,
 \]
 which produce the updates
 \[
 \mX^{t+1} = \mX^t - \gamma_t \mU \frac{\diag\left(\sigma_1^{q-1}, \dots, \sigma_{\min(m,n)}^{q-1}\right)}{\left(\sum_{i=1}^{\min(m,n)}{\sigma_i^q}\right)^{\frac{q-1}{q}}}\mV^\top
 \]
 where $p$ and $q$ satisfy $p^{-1} + q^{-1} = 1$. This formula recovers Neon when $p\rightarrow1$ (provided that $\sigma_1 > \sigma_2$, which holds with probability 1 on real data), NSGD when $p=2$, and Muon when $p\rightarrow\infty$.
 
 However, Schatten norms do not fill the rank gap: when $p > 1$, the update has full rank, while when $p=1$, its rank equals 1. Moreover, computing the update for $p \neq 1, 2, \infty$ appears to require knowing all $\sigma_i$, making the problem as hard as computing the full SVD.
 
 \paragraph{Ky Fan norms}
 Another family of matrix norms could offer a solution: Ky Fan norms. For $k \in \{1, \dots, \min(m, n)\}$, the Ky Fan $k$-norm $\normkfk{\cdot}$ equals $\sum_{i=1}^{k}\sigma_i$, the sum of the $k$ largest singular values. Notable special cases include the Ky Fan $1$-norm (the spectral norm) and the Ky Fan $\min\{m, n\}$-norm (the nuclear norm).
 
 To derive the update formula for arbitrary $k$, we use the expression for the norm dual to the Ky Fan $k$-norm (see, for instance, \citep{bhatia2013matrix}, p. 96):
 \[
     \normkfk{\cdot}^\dagger = \max\left\{\frac{1}{k} \normn{\cdot}, \norms{\cdot}\right\}\,.
 \]
 According to (\ref{eq:our_update}), we need to find $\mD$ such that $\normkfk{\mD} = 1$ and
 \[
 \<\mM^t, \mD> = \max\left\{\frac{1}{k}\sum_{i=1}^{\min(m,n)}\sigma_i,\ \sigma_1\right\}\,.
 \]
 Neon's update $\mD = \vu_1 \vv_1^\top$ achieves $\sigma_1$, while Muon's scaled update $\mD = \frac{1}{k}\mU\mV^\top$ achieves $\frac{1}{k}\sum_{i=1}^{\min(m,n)}\sigma_i$. Thus, the update is either
 \[
 \mX^{t+1} = \mX^{t} - \gamma_t \vu_1 \vv_1^\top\text{ or }\mX^{t+1} = \mX^{t} - \frac{\gamma_t}{k}\mU \mV^\top\,,
 \]
 depending on which of those two expressions achieves a smaller residual on the linear function $L(\mX^{t+1}) := F(\mX^t) + \<\mM^t, \mX^{t+1} - \mX^t>$. The Ky Fan norms thus fail to close the rank gap: the resulting update is either rank-one or full-rank.
 
 \subsection{Solution: Duals to Ky Fan Norms}
 
 Unlike Schatten norms, which satisfy $\norm{\mM^t}_{S_p}^\dagger = \norm{\mM^t}_{S_q}$ (for $p^{-1} + q^{-1} = 1$) and are thus closed under dualization, Ky Fan norms are generally not. Only two exceptional cases exist: the dual to the Ky Fan $1$-norm $\norm{\cdot}_{\text{KF-}1}$ (the spectral norm) is the Ky Fan $\min(m,n)$-norm $\norm{\cdot}_{\text{KF-}\min(m,n)}$ (the nuclear norm), and vice versa. Thus, working with duals of Ky Fan norms can open up new possibilities:
 
 \begin{lemma}\label{lemma:ky_fan_update}
   When $\norm{\cdot} = \normkfk{\mM^t}^\dagger$, (\ref{eq:our_update}) turns into:
   \begin{equation}\label{eq:ky_fan_update}
     \mX^{t+1} = \mX^{t} - \gamma_t \sum_{i=1}^{k}\vu_i \vv_i^\top\,.
   \end{equation}
 \end{lemma}
 
 \begin{proof}
 Since $\norm{\cdot}^\dagger = \normkfk{\cdot}^{\dagger\dagger} = \normkfk{\cdot}$, the goal is to reach $\<\mM^t, \mD> = \normkfk{\mM^t}$ in (\ref{eq:our_update}). Note that $\mD = \sum_{i=1}^{k} \vu_i \vv_i^\top$ attains this value. Indeed,
 \[
 \<\mM^t, \mD> =\<\mU \mSigma \mV^\top, \sum_{i=1}^{k} \vu_i \vv_i^\top> = \sum_{i,j=1}^{r, k}\<\vu_i \sigma_i \vv_i^\top, \vu_j \vv_j^\top> = \sum_{i=1}^{k}\sigma_i = \normkfk{\mM^t}\,,
 \]
 which completes the proof.\qed
\end{proof}
 
 These updates define the \emph{Fanion} family of LMO-based algorithms, each operating under a norm $\normkfk{\mM^t}^\dagger$. We denote the algorithm for a particular $k$ as \emph{Fanion-$k$}. This family elegantly bridges the rank gap, providing updates of any intermediate rank $k$. For a visualization of the ball in the $\normkfk{\cdot}^\dagger$ norm, see \sectionref{subsec:kyfan2norm}.
 
 \paragraph{Connection to existing algorithms}
 The Fanion family unifies several known algorithms:
 \begin{itemize}
   \item \textbf{Neon} is Fanion-1 (rank-one updates)
   \item \textbf{Muon} is Fanion-$\min\{m,n\}$ (full-rank updates)
   \item \textbf{Dion} (unsharded): The rank-$r$ Dion (Algorithm 1 from \citep{ahn2025dion}) without error feedback and without scaling of the update is actually Fanion-$r$ (see \citep{pethick2025understandingdion}, where Dion is written in a notation more similar to ours)
 \end{itemize}
 
 In \sectionref{sec:matrix-side}, we discuss how to efficiently compute Fanion updates using the Lanczos algorithm.
 
 \section{Conic Combination of LMO-Algorithms Is an LMO-Algorithm}
 \label{sec:conic_combo}
 
 The approaches to designing new matrix LMO-algorithms are not limited to applying norms dual to Ky Fan $k$-norms. We now consider linear combinations of LMO-algorithms and show that these combinations are themselves LMO-algorithms.
 
 \subsection{General Case}
 We begin with a known property of a conic combination of norms (see, e.g., \citep[Table 1]{yu2012arithmetic}).
 
 \begin{lemma}\label{lemma:dual_to_conv_comb}
   Let $\norm{\cdot}_{(1)}, \dots, \norm{\cdot}_{(n)}$ be norms on a finite-dimensional Euclidean space, and let $\alpha_1, \dots, \alpha_n$ be non-negative reals. Define
   \[
     \norm{\cdot} := \sum_{i = 1}^n \alpha_i \norm{\cdot}_{(i)}\,.
   \]
   Then the dual unit ball of $\norm{\cdot}$ satisfies
   \[
     \cB_{\norm{\cdot}^\dagger}
     = \sum_{i = 1}^n \alpha_i \cB_{\norm{\cdot}_{(i)}^\dagger}\,,
   \]
   where $\sum$ denotes the Minkowski sum and $\cB_{\norm{\cdot}_{(i)}^\dagger}$ is the unit ball of the dual norm $\norm{\cdot}_{(i)}^\dagger$.
 \end{lemma}
 
 A proof is provided in the appendix (see \sectionref{sec:proof_dual_conv_comb}).
 
 \begin{lemma}\label{lemma:lin_comb_lmo}
  Let $\norm{\cdot}_{(1)}, \dots, \norm{\cdot}_{(n)}$ be norms on a finite-dimensional Euclidean space, and let $\alpha_1, \dots, \alpha_n$ be non-negative reals. Consider Linear Minimization Oracles $\mathrm{LMO}_{\norm{\cdot}_{(1)}}, \dots, \mathrm{LMO}_{\norm{\cdot}_{(n)}}$, which correspond to the unit balls of these norms. Then, $\sum_{i = 1}^n \alpha_i \mathrm{LMO}_{\norm{\cdot}_{(i)}}$ is the LMO corresponding to the norm $\norm{\cdot}$ dual to the norm given by $\sum_{i = 1}^n \alpha_i \norm{\cdot}_{(i)}^\dagger$.
\end{lemma}
 
 \begin{proof}
   Using \lemmaref{lemma:dual_to_conv_comb} and the biduality property $\norm{\cdot}^{\dagger \dagger} = \norm{\cdot}$, we obtain the unit ball representation: $\cB_{\norm{\cdot}} = \sum_{i = 1}^n \alpha_i \cB_{\norm{\cdot}_{(i)}}$. The linear minimization problem over this ball can thus be transformed as follows:
   \[
   \argmin_{\mD \in \cB_{\norm{\cdot}}}\<\mM, \mD> = \argmin_{\mD_1 \in \alpha_1 \cB_{\norm{\cdot}_{(1)}}, \dots, \mD_n \in \alpha_n\cB_{\norm{\cdot}_{(n)}}}\<\mM, \sum_{i = 1}^n\mD_i> = \sum_{i = 1}^n \argmin_{D_i \in \cB_{\norm{\cdot}_{(i)}}}\<\mM, \mD_i>\,,
   \]
   where the last summation denotes the Minkowski sum. This immediately implies 
   \[
   \sum_{i = 1}^n \alpha_i \mathrm{LMO}_{i} \in \argmin_{\mD \in \cB_{\norm{\cdot}}}\<\mM, \mD>\,,
   \] 
   completing the proof.\qed
\end{proof}
 
 Applying this result to optimization algorithms yields the following corollary.
 
 \begin{corollary}\label{corollary:lin_comb_alg}
  Let there be a finite family of LMO-based algorithms indexed by $i = 1, \dots, n$, where the $\mX^{t + 1}$ update of the $i$-th algorithm is defined by
  \[
  \mX^{t + 1} - \mX^{t} = \gamma_t \mathrm{LMO}_{\norm{\cdot}_{(i)}}(\mM^{t})\,,
  \]
  and $\mathrm{LMO}_{\norm{\cdot}_{(i)}}$ corresponds to the unit ball of norm $\norm{\cdot}_{(i)}$. Then, for arbitrary non-negative $\alpha_1, \dots, \alpha_n$, the algorithm with the update given by
  \[
  \mX^{t + 1} - \mX^{t} = \gamma_t \sum_{i = 1}^n \alpha_i \mathrm{LMO}_{\norm{\cdot}_{(i)}}(\mM^t)
  \]
  is an LMO-algorithm itself, with LMO corresponding to the unit ball of the norm $\norm{\cdot}$ dual to the norm given by $\sum_{i = 1}^n \alpha_i \norm{\cdot}_{(i)}^\dagger$.
\end{corollary}
 
 \subsection{Frobeniusize Them: F-Muon and F-Neon}
 We now construct concrete examples of algorithms obtained via linear combinations of LMO-algorithms. By \corollaryref{corollary:lin_comb_alg}, these combinations are themselves LMO-algorithms.
 
 Combining Fanions with NSGD yields a family of algorithms with updates
 \begin{equation}\label{eq:nsgd_muon_update}
     \mX^{t+1} = \mX^{t} - \gamma_t \left(\alpha \sum_{i = 1}^k \vu_i \vv_i^\top + (1-\alpha)\frac{\mM^t}{\normf{\mM^t}}\right)\,.
 \end{equation}
 Recall that Fanion-$k$ operates under the dual to the Ky Fan $k$-norm, while NSGD operates under the self-dual Frobenius norm. By \corollaryref{corollary:lin_comb_alg}, this combination defines an LMO-algorithm with norm $\norm{\cdot}_{\mathrm{F-KF-k}}^\dagger$, where
 \begin{equation}\label{eq:fkfk_norm}
   \norm{\cdot}_{\mathrm{F-KF-k}} = \alpha \normkfk{\cdot} + (1-\alpha)\normf{\cdot}\,.
 \end{equation}
 We call this family \emph{F-Fanions}.
 
 The extreme members of this family are \emph{F-Neon} (with $k = 1$) and \emph{F-Muon} (with $k = \min\{m, n\}$). Additional information and visualizations for the F-Muon generating $\normfstar{\cdot}^\dagger = \norm{\cdot}_{\mathrm{F-KF-\min\{m, n\}}}^\dagger$ and the F-Neon generating $\normftwo{\cdot}^\dagger = \norm{\cdot}_{\mathrm{F-KF-1}}^\dagger$ norms appear in the appendix (see (\ref{eq:normfstardual}), (\ref{eq:normftwodual}), \figureref{fig:fduals}).
 
 \subsection{Add SignSGD: S-Muon and S-Neon}
 
 Similarly, combining Fanion-$k$ with SignSGD yields \emph{S-Fanion-$k$} with the update
 \begin{equation}\label{eq:sign_muon_update}
     \mX^{t+1} = \mX^{t} - \gamma_t \left(\alpha \sum_{i = 1}^k \vu_i \vv_i^\top + (1-\alpha)\eta\sign(\mM^t)\right)\,,
 \end{equation}
 where \texttt{sign\_lr\_coeff} $\eta$ is a scaling coefficient for SignSGD.
 
 By \corollaryref{corollary:lin_comb_alg}, this defines an LMO-algorithm with norm $\norm{\cdot}_{\mathrm{C^\dagger-KF-k}}^\dagger$, where
 \begin{equation}\label{eq:ckfk_norm}
   \norm{\cdot}_{\mathrm{C^\dagger-KF-k}} = \alpha \normkfk{\cdot} + \frac{1-\alpha}{\eta}\normc{\cdot}^\dagger\,.
 \end{equation}
 
 The extreme members of this family are \emph{S-Neon} (with $k = 1$) and \emph{S-Muon} (with $k = \min\{m, n\}$).
 
 \section{Computing the Updates}\label{sec:matrix-side}
 We employ the thick-restart Lanczos method for the symmetric eigenvalue problem (TRLan) to compute the low-rank updates of Fanions. We apply TRLan to either $\mM^{t\top}\mM^t$ or $\mM^t \mM^{t\top}$, selecting whichever matrix is smaller. We use the CuPy implementation of \texttt{cupy.sparse.linalg.svds} \citep{cupy_svds_ref}, which internally relies on TRLan \citep{simonz1998thick}.
 
 TRLan is specifically designed for efficiently approximating the largest singular values and corresponding singular vectors of very large matrices. The thick-restart strategy retains the most informative Ritz vectors across restarts, which dramatically accelerates convergence while maintaining moderate memory consumption. TRLan is particularly well-suited to our GPU setting because its dominant computational cost consists of a modest number of highly parallelizable matrix-vector multiplications (matvecs), and it avoids full reorthogonalization against the entire Krylov basis by employing short recurrence relations combined with thick restarting.
 
 The per-cycle complexity is $\mathcal{O}(mn^2 + n^2d + nd^2)$, where $m \geq n$ are the dimensions of the target matrix and $d$ is the size of the retained subspace ($d \ll n$ typically).
 
 In \tableref{tbl:matrix_methods}, we compare TRLan against randomized SVD (RSVD) and simple power iterations for computing the rank-$k$ update used in Fanion-$k$ and related algorithms. Experiments are performed on dense random matrices with i.i.d. $\mathcal{N}(0,1)$ entries using CPU implementations for fair comparison.
 We report:
 \begin{itemize}
   \item $err_1$: relative error in the Frobenius norm of $\sum_i^k \vu_i \vv_i^T$,
   \item $err_2$: relative error in the Frobenius norm of $\sum_i^k \sigma_i \vu_i \vv_i^T$.
 \end{itemize}
 
 On $500\times500$ matrices, TRLan and RSVD require comparable wall-clock time, but TRLan delivers orders-of-magnitude lower error with far fewer matvecs. On larger $5000\times5000$ matrices, this advantage becomes even more pronounced: TRLan is 3-4 times faster than RSVD while using $\sim$30 times fewer matvecs at comparable or superior accuracy.
 
 An interesting empirical observation is that RSVD tends to approximate the \emph{singular values} themselves reasonably well, but the reconstructed low-rank matrix exhibits noticeable deviation from the truncated SVD. In contrast, TRLan provides an excellent approximation to the truncated SVD matrix itself (low $err_2$), albeit at the cost of occasionally less accurate individual singular values. This makes TRLan the preferred choice for algorithms like Neon/Fanion-$k$ that only require the low-rank term $\sum \sigma_i \vu_i \vv_i^\top$, but less suitable for methods (e.g., Dion) that explicitly require accurate $\sigma_i$ for error feedback or step-size control.
 
 A current practical limitation is the absence of a native PyTorch implementation of thick-restart Lanczos; existing PyTorch-based randomized SVD routines cannot match TRLan's accuracy-efficiency combination for the matrix reconstruction task.
 
 For reference, \tableref{tbl:newton_schulz_reference} presents results for the Newton-Schulz polar decomposition iteration on the same matrices, where $err_1$ is the relative error of $\mU\mV^T$ (29-30 iterations to converge, resulting in a significantly higher matvec count than TRLan).
 
 \begin{table}[ht]
 
\caption{Comparison of methods for computing rank-$k$ updates on dense random matrices (CPU, double precision). Lower is better in all columns.}\label{tbl:matrix_methods}
\resizebox{\textwidth}{!}{%
     \begin{tabular}{cccccccc}
       \arrayrulecolor{black}\toprule
       Matrix sizes & $k$  & Method           & Time (s) & Matvecs & Iterations & $err_1$         & $err_2$       \\
       \arrayrulecolor{black}\hline
       $500 \times 500$      & 5  & Power Iterations           & 0.062   & 2005    & 200        & $9.2 \cdot 10^{-3}$ & $9.1 \cdot 10^{-3}$\\
       $500 \times 500$      & 5  & RSVD             & 0.017   & 1170    & 38         & $9.8 \cdot 10^{-3}$ & $9.6 \cdot 10^{-3}$\\
       $500 \times 500$      & 5  & TRLan            & 0.018   & 131     & 65         & $9.6 \cdot 10^{-5}$ & $9.4 \cdot 10^{-5}$\\
       \hdashline
       $500 \times 500$      & 50 & Power Iterations           & 0.44    & 43750   & 437      & $9.9 \cdot 10^{-3}$ & $9.0 \cdot 10^{-3}$\\
       $500 \times 500$      & 50 & RSVD             & 0.61    & 6120    & 50      & $9.9 \cdot 10^{-3}$ & $9.1 \cdot 10^{-3}$\\
       $500 \times 500$      & 50 & TRLan            & 0.16    & 462     & 231      & $3.3 \cdot 10^{-7}$ & $3.0 \cdot 10^{-7}$\\
       \hdashline
       $5000 \times 5000$    & 5  & Power Iterations           & 9.6     & 9065    & 906      & $8.6 \cdot 10^{-3}$ & $8.6 \cdot 10^{-3}$\\
       $5000 \times 5000$    & 5  & RSVD             & 2.1     & 5640    & 187      & $9.7 \cdot 10^{-3}$ & $9.7 \cdot 10^{-3}$\\
       $5000 \times 5000$    & 5  & TRLan            & 0.70    & 205     & 102      & $7.7 \cdot 10^{-3}$ & $7.7 \cdot 10^{-3}$\\
       \arrayrulecolor{black}\bottomrule
     \end{tabular}%
   }

 \end{table}
 
 \begin{table}[ht]
 
\caption{Newton-Schulz iterations performance on random dense matrices (for reference).}\label{tbl:newton_schulz_reference}
     \begin{tabular}{ccccc}
       \arrayrulecolor{black}\toprule
       Matrix size    & Time (s) & Matvecs & Iterations & $err_1$ \\
       \arrayrulecolor{black}\hline
       $500 \times 500$     & 0.041    & 27 000  & 27         & $4.8 \cdot 10^{-3}$  \\
       $5000 \times 5000$   & 26.4     & 290 000 & 29         & $6.5 \cdot 10^{-3}$  \\
       \arrayrulecolor{black}\bottomrule
     \end{tabular}%

 \end{table}
 
 \section{Experiments}\label{sec:experiments}
 
 \subsection{A Smooth Convex Problem}
 \label{subsec:lls_exps}
 
 We begin by evaluating F-Fanions and S-Fanions on the following $L$-smooth convex problem:
 \begin{equation}\label{eq:lls}
   F(\mX) = \frac{1}{2}\normf{\mM^{1/2}\mX\mN^{1/2}}^2 = \frac{1}{2} \<\mX, \mM \mX \mN> \to \min_{\mX \in \Rmn}\,,
 \end{equation}
 where $\mX \in \Rmn$ with $m=500$ and $n=500$, corresponding to the typical dimensions of a neural network weight matrix, while $\mM \in \mathbb{S}^m_{+}$ and $\mN \in \mathbb{S}^n_{+}$ are positive-semidefinite matrices. The spectra of $\mM$ and $\mN$ are uniformly distributed on the interval $(0, 1)$, so $L \le 1$ for the Frobenius norm. We initialize $\mX^0$ with entries drawn independently from $\cN(0, 0.01)$.
 
 We evaluate several algorithms from the Fanion family and their convex combinations, using the update rule from (\ref{eq:simple_update}) with approximate Nesterov momentum:
 \begin{itemize}
   \item \textbf{Fanions:} Neon (Fanion-1), Fanion-2, Fanion-10, Fanion-100, and Muon (Fanion-500).
   \item \textbf{F-Fanions:} F-Neon (F-Fanion-1), F-Fanion-2, F-Fanion-10, F-Fanion-100, and F-Muon (F-Fanion-500) with $\alpha=1/2$.
   \item \textbf{S-Fanions:} S-Neon (S-Fanion-1), S-Fanion-2, S-Fanion-10, S-Fanion-100, and S-Muon (S-Fanion-500) with $\alpha=1/2$ and \texttt{sign\_lr\_coeff=0.01}.
 \end{itemize}
 We also include the baselines of Normalized SGD and SignSGD, which correspond to F-Fanion and S-Fanion, respectively, with $\alpha = 0$ and arbitrary $k$.
 
 Since theoretical bounds \citep{kovalev2025understanding, riabinin2025gluon} rely on a loose norm bound $\norm{\cdot}\le \rho \normf{\cdot}$, we do not derive the learning rate or Nesterov momentum from the smoothness constants, which also depend on the choice of the norm. Instead, we identify the \texttt{(lr, momentum)} pair for which the optimizer reaches the loss threshold of $0.001$ in the fewest iterations. This setting is both realistic and consistent with the corollaries of convergence theorems that propose constant learning rate and momentum coefficients.
 
 We perform a grid search over the following hyperparameter ranges:
 \begin{itemize}
   \item \texttt{momentum} $\in \{0.1, 0.2, 0.3, 0.4, 0.5, 0.6, 0.7, 0.8, 0.9, 0.95\}$ for all algorithms,
   \item \texttt{lr} $\in [0.005, 0.020]$ with step $0.001$ for Muon, F-Muon, and S-Muon,
   \item \texttt{lr} $\in [5 \cdot 10^{-5},\, 2 \cdot 10^{-4}]$ with step $1 \cdot 10^{-5}$ for SignSGD,
   \item \texttt{lr} $\in [0.01, 0.10]$ with step $0.01$ for NSGD.
 \end{itemize}
 The tuned parameters and the number of iterations required to converge to $0.001$ are presented in \tableref{tbl:lls_lrs}. For intermediate-rank Fanions, F-Fanions, and S-Fanions, the hyperparameters \texttt{lr}, \texttt{momentum}, and \texttt{sign\_lr\_coeff} are set equal to those of Muon, F-Muon, and S-Muon, respectively, as their losses decrease too slowly to reach $0.001$ within a reasonable tuning budget.
 
 \begin{table}[htbp]
 
\caption{Tuned learning rates and momentum coefficients in the experiments on the smooth convex problem (\ref{eq:lls}).}\label{tbl:lls_lrs}
     \begin{tabular}{cccc}
       \arrayrulecolor{black}\toprule
       Algorithm & \texttt{lr}  & \texttt{momentum}  & Iterations to 0.001 loss \\
       \arrayrulecolor{black}\hline
       Muon & 0.007 & 0.5 & 1060 \\
       \hdashline
       NSGD & 0.08 & 0.95 & 1020 \\
       F-Muon & 0.015 & 0.7 & 910 \\
       \hdashline
       SignSGD & 0.016 $\times$ 0.01 & 0.95 & 2650 \\
       S-Muon & 0.011 & 0.9 & 890 \\
       \arrayrulecolor{black}\bottomrule
     \end{tabular}%

 \end{table}
 
 The results are presented in \figureref{fig:lls} with additional details in \sectionref{sec:lls_plots_section}.
 
 \begin{figure}[htbp]

     \subfigure[Loss]{\label{fig:lls_loss}
       \includegraphics[width=0.8\linewidth]{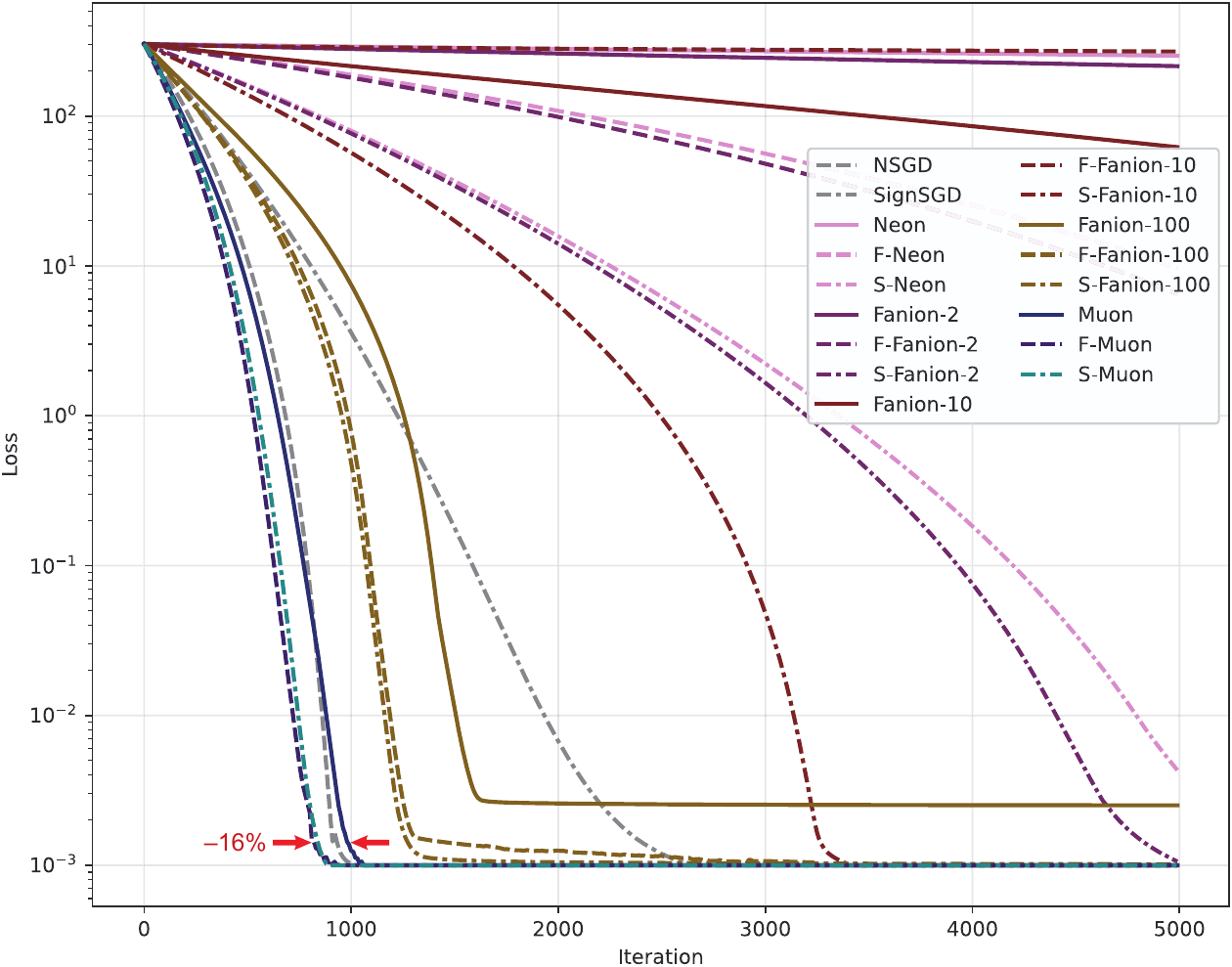}}%
     \vspace{1em}
     \subfigure[Frobenius norm of the full gradient]{\label{fig:lls_fro_grad_norm}%
       \includegraphics[width=0.8\linewidth]{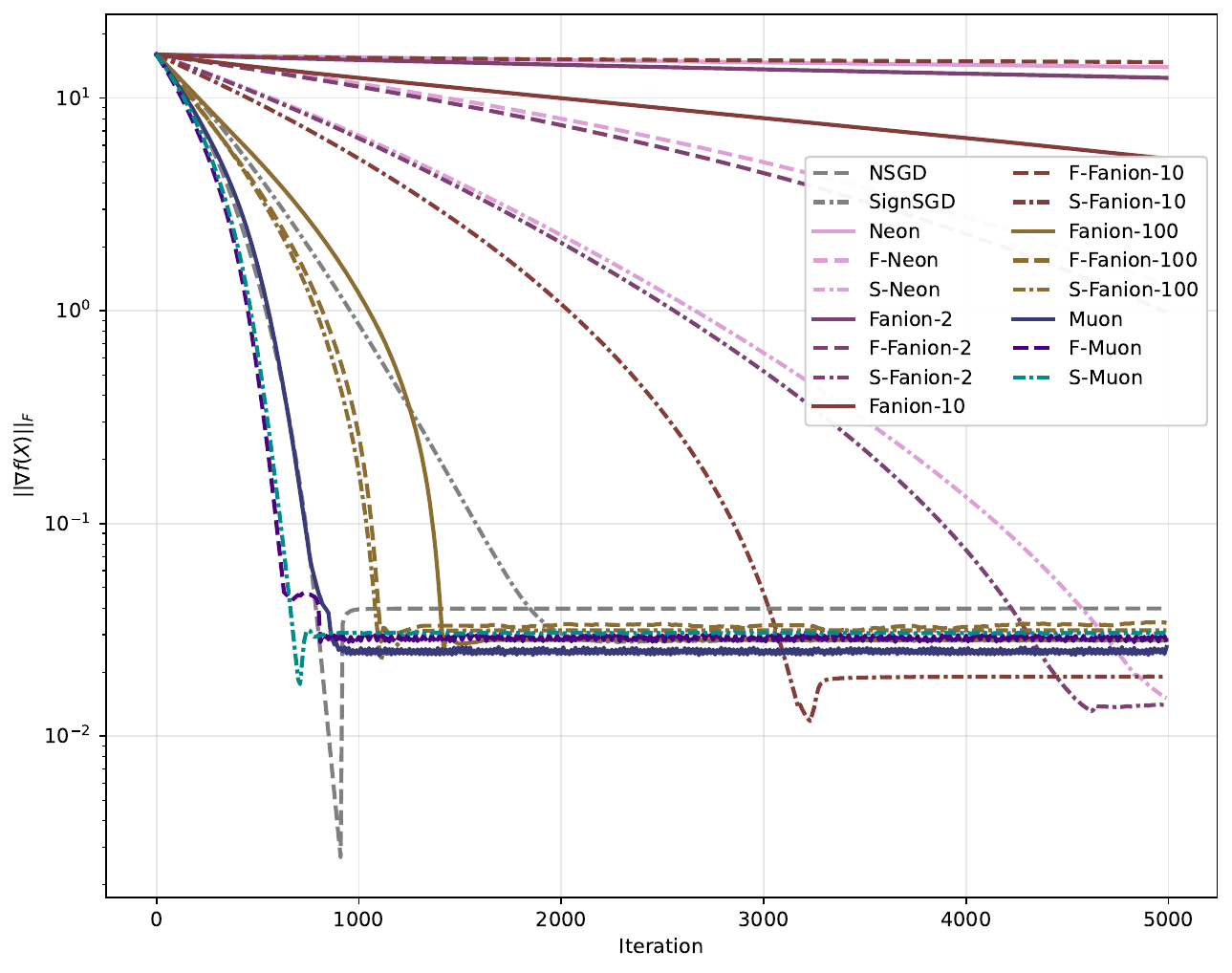}}
   
\caption{A smooth convex problem (\ref{eq:lls}) for a 500x500 matrix.}\label{fig:lls}

 \end{figure}
 
 Both F-Muon and S-Muon converge faster to lower loss values and achieve lower Frobenius norms of the full gradient than NSGD, Muon, or SignSGD.
 
 \subsection{CIFAR-10 Airbench}
 
 We evaluate the algorithms on the CIFAR-10 airbench \citep{cifar2023airbench}. To assess the impact of the mixing parameter $\alpha$ in (\ref{eq:nsgd_muon_update}), we first run F-Muon for different values of $\alpha$ using hyperparameters tuned for vanilla Muon by Keller Jordan: \texttt{lr=0.24(1 - step/total\_steps)}, \texttt{momentum=0.65}, \texttt{nesterov=True} with weight normalization after each weight update. We perform 10 repetitions for each $\alpha$ value and record the accuracy after 8 epochs of training (\figureref{fig:muon_alphas}).
 
 Next, we tune F-Muon specifically with $\alpha = 0.5$, obtaining optimal hyperparameters \texttt{lr=0.4(1 - step/total\_steps)}, \texttt{momentum=0.6}, \texttt{nesterov = True}. Tuned F-Muon achieves $94.02 \pm 0.13\%$ validation accuracy after 8 epochs (averaged over 200 runs), matching Muon's accuracy and variance. We again measure the validation accuracy as a function of $\alpha$ (\figureref{fig:fmuon_alphas}) and observe that even at $\alpha=0.1$, the accuracy substantially exceeds that of vanilla NSGD.
 
 For S-Muon with $\alpha=0.5$, we tune all hyperparameters to find the optimal configuration \texttt{lr=0.42(1 - step/\allowbreak{}total\_steps)}, \texttt{momentum=0.63}, \texttt{nesterov = True}, \texttt{sign\_lr\_coeff = 0.003}, achieving $94.03 \pm 0.13\%$ validation accuracy (\figureref{fig:diff_alphas}), which slightly exceeds vanilla Muon's performance.
 
 \begin{figure}[htbp]

     \subfigure[With parameters tuned for Muon]{\label{fig:muon_alphas}%
       \includegraphics[width=0.48\linewidth]{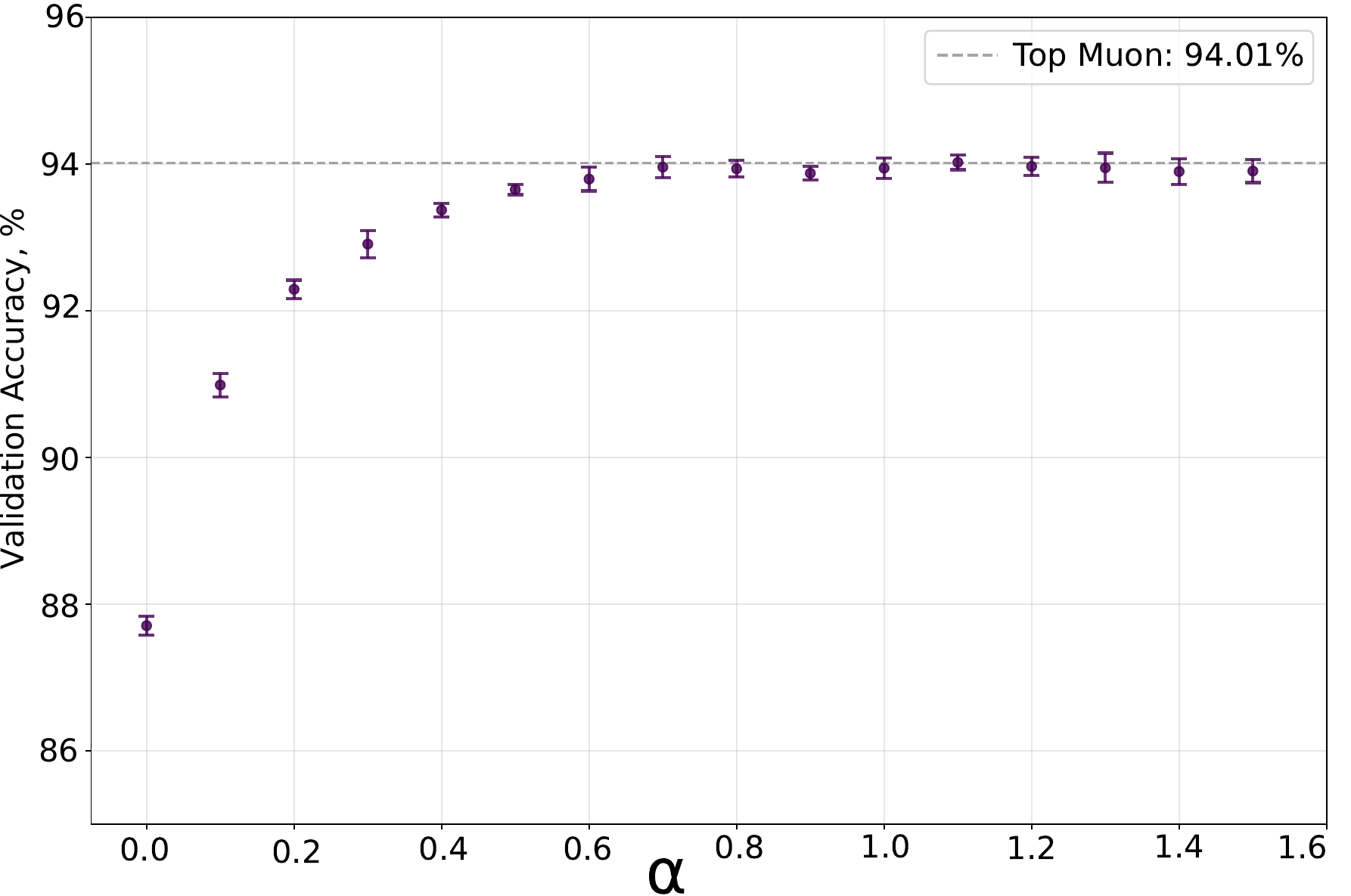}}%
     \hfill
     \subfigure[With parameters tuned for F-Muon]{\label{fig:fmuon_alphas}%
       \includegraphics[width=0.48\linewidth]{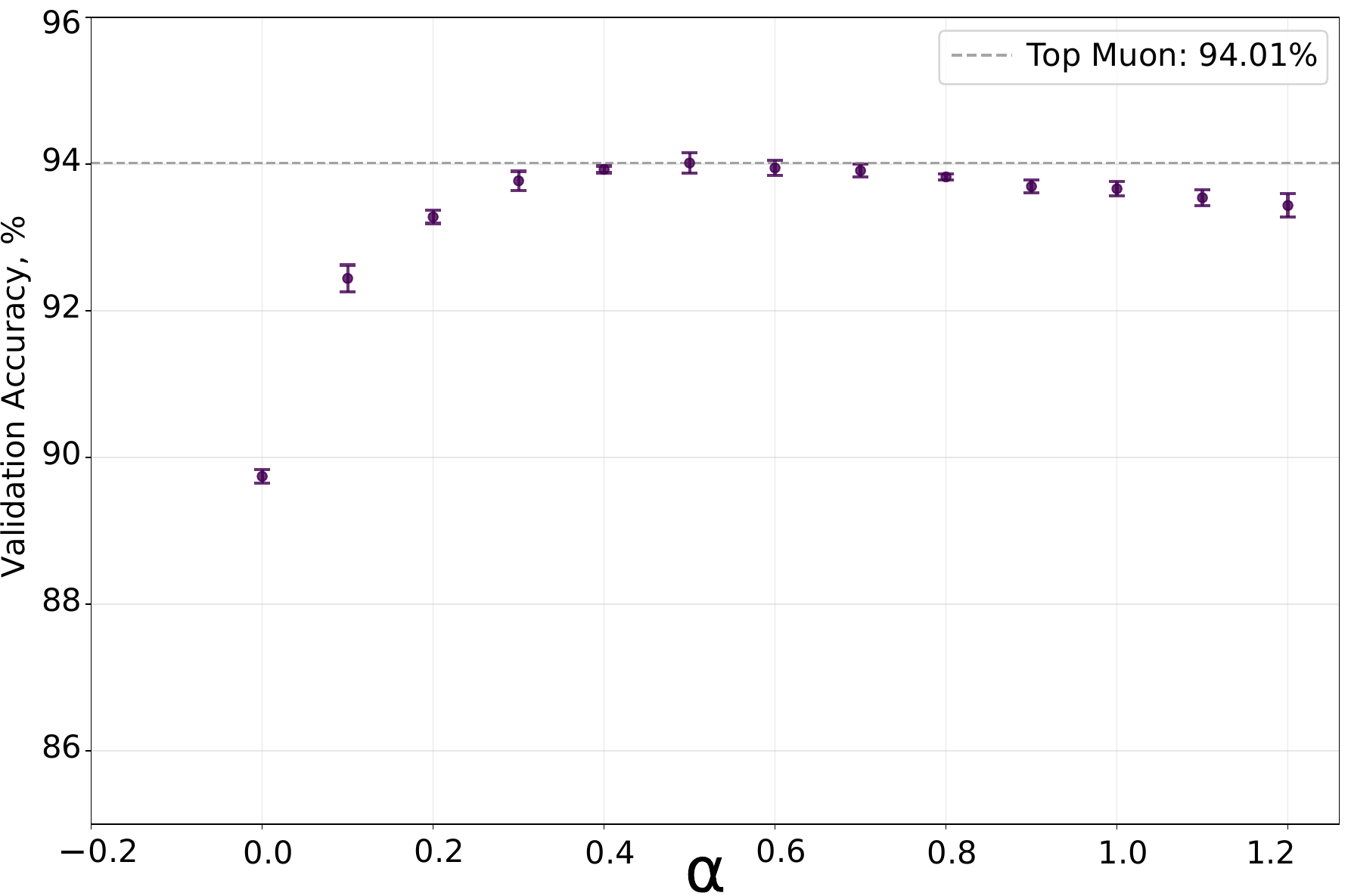}}
   
\caption{Mean validation accuracies for F-Muon with different $\alpha$.}\label{fig:diff_alphas}

 \end{figure}
 
 The Muon-level performance is remarkable given that F-Muon and S-Muon operate with substantially different constraint geometries. \figureref{fig:cifar_ball} visualizes F-Muon by plotting the LMO balls of Muon and F-Muon (scaled by actual learning rates) in the 2D space of singular values (with all other singular values set to zero). S-Muon cannot be visualized in this manner because the Chebyshev norm is not a function of singular values alone, and visualizing $\norms{\cdot}$ for $\Rmn$ with $m, n \geq 2$ requires at least a 4D space. Nevertheless, the LMO ball of S-Muon clearly differs substantially from Muon's, as the effective SignSGD learning rate $\texttt{lr} = (1-\alpha)\texttt{sign\_lr\_coeff} = 6.3 \cdot 10^{-4}$ is comparable to typical SignSGD learning rates in deep learning.
 
 \begin{figure}[htbp]
 
\includegraphics[width=0.6\linewidth]{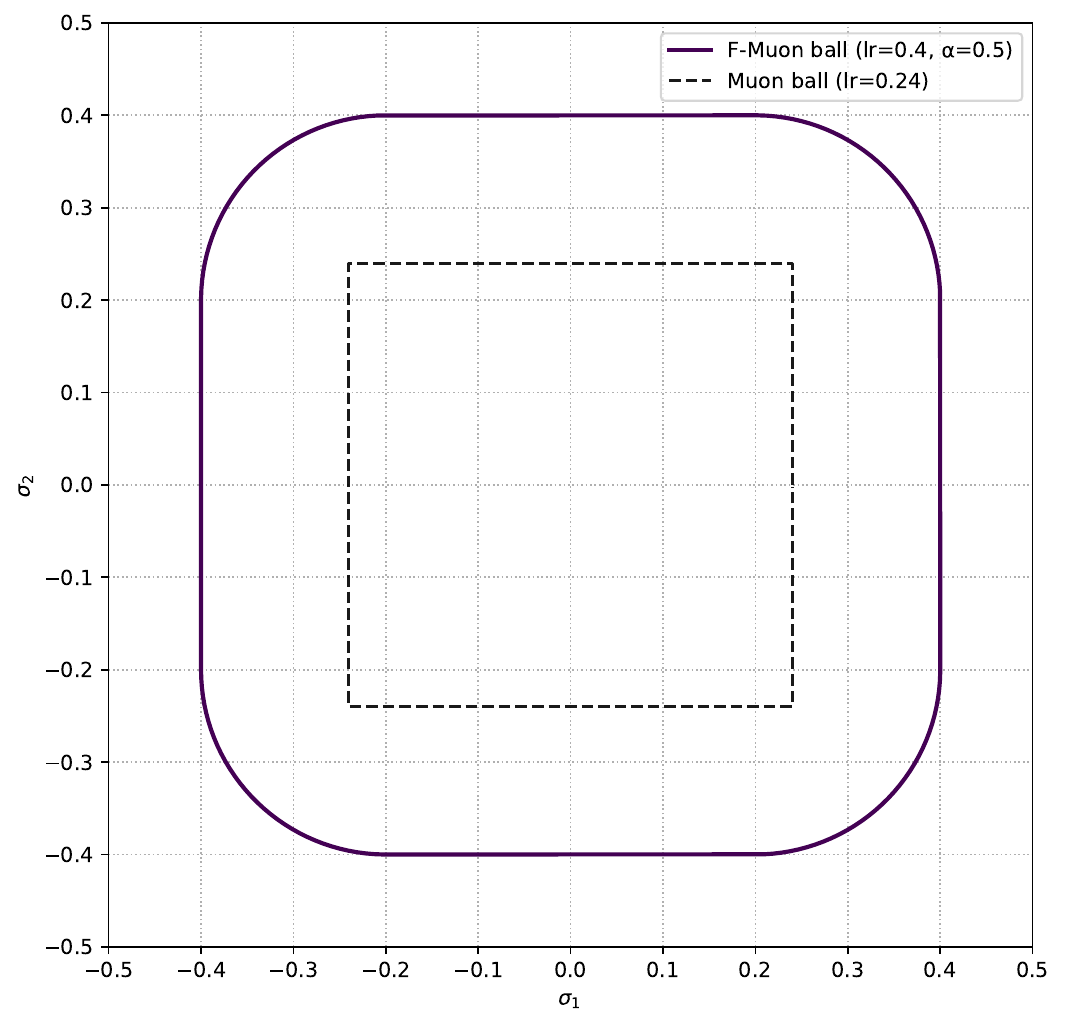}
\caption{The LMO balls of Muon and F-Muon for training a CNN on CIFAR-10.}\label{fig:cifar_ball}

 \end{figure}
 
 Notably, even the pathological case $\alpha > 1$, which corresponds to an LMO constraint region that is not a norm ball, achieves accuracy nearly identical to vanilla Muon. These observations raise a fundamental question about the sensitivity of LMO-based algorithms to the shape of the constraint region.
 
 Additional results, including gradient norm profiling and performance comparisons with Fanions and F-Fanions, are provided in \sectionref{sec:cifar_underhood}.
 
 \subsection{NanoGPT Speedrun}
 
 We evaluate F-Muon and S-Muon with $\alpha=0.5$ against Muon, SignSGD, and NSGD on the NanoGPT speedrun benchmark \citep{modded_nanogpt_2024}. The optimal hyperparameters are shown in the legend of \figureref{fig:nano_gpt_loss}, with \texttt{sign\_sgd\_coeff} values of $3 \cdot 10^{-4}$ for S-Muon. After 1750 training steps, F-Muon achieves a cross-entropy loss of 3.281, S-Muon achieves 3.287, and Muon achieves 3.279, falling below the target threshold of 3.280. As illustrated in \figureref{fig:nano_gpt_loss}, these differences are negligible. Notably, F-Muon, a convex combination of Muon and NSGD, is efficient despite NSGD performing poorly in isolation. The same observation holds for S-Muon and SignSGD.
 
 \begin{figure}[htbp]
 
\includegraphics[width=0.7\linewidth]{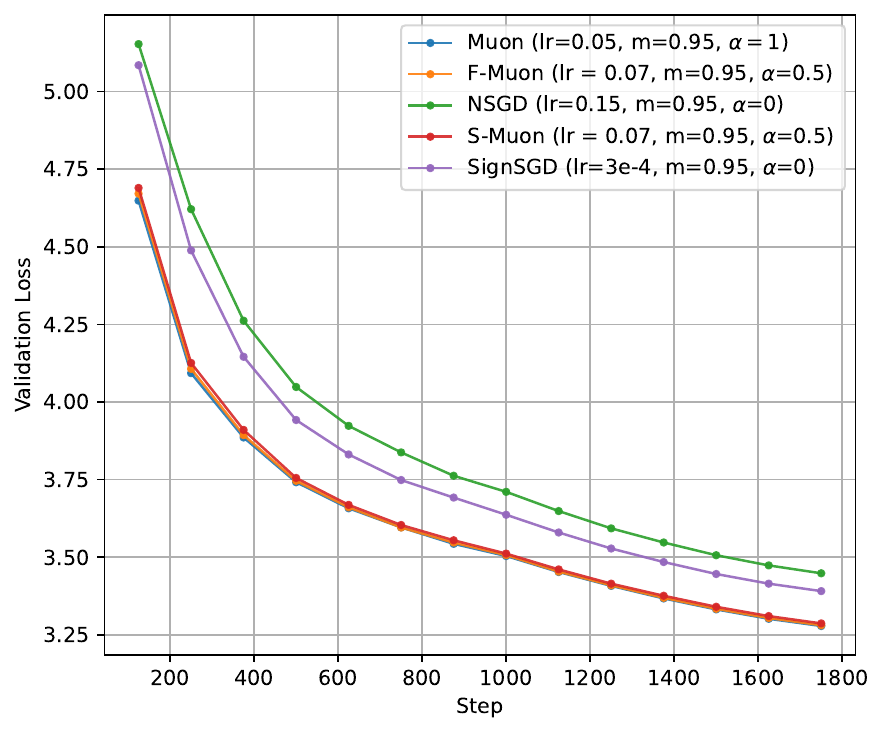}
\caption{Validation loss for NanoGPT.}\label{fig:nano_gpt_loss}

 \end{figure}
 
 As observed in the CIFAR-10 airbench experiments, setting $\alpha=1.1$ for F-Muon (corresponding to the no-ball configuration) yields a loss of 3.2818, representing only a marginal difference from the baseline.
 
 Interestingly, F-Muon with the NSGD component \textit{not scaled} by Muon's $\sqrt{m/n}$ rule yields a loss of 3.288 (an increase of 0.007), while S-Muon \textit{scaled} with this rule achieves a loss of 3.292 (an increase of 0.005). Thus, there is no clear rule as to whether the non-Fanion update part should be scaled by $\sqrt{m/n}$.
 
 \subsection{GPT-2 Medium Speedrun}
 
 We scale from NanoGPT to GPT-2 Medium (24 transformer layers, 1024-dimensional hidden layers, 16 attention heads, approximately 345 million parameters), evaluating the same algorithms with $\alpha=0.5$ on the FineWeb dataset. After 5960 training steps, Muon achieves a validation loss of 2.9198, successfully reaching the speedrun threshold of 2.92. F-Muon achieves 2.9215, and S-Muon achieves 2.9235, narrowly missing the threshold. As illustrated in \figureref{fig:med_gpt_loss}, the performance gap remains remarkably small: F-Muon trails Muon by only 0.0017, while S-Muon trails by 0.0037. These minimal differences demonstrate that alternative norm constraints maintain competitive performance even at this significantly larger model scale.
 
 \begin{figure}[htbp]
   
\includegraphics[width=0.7\linewidth]{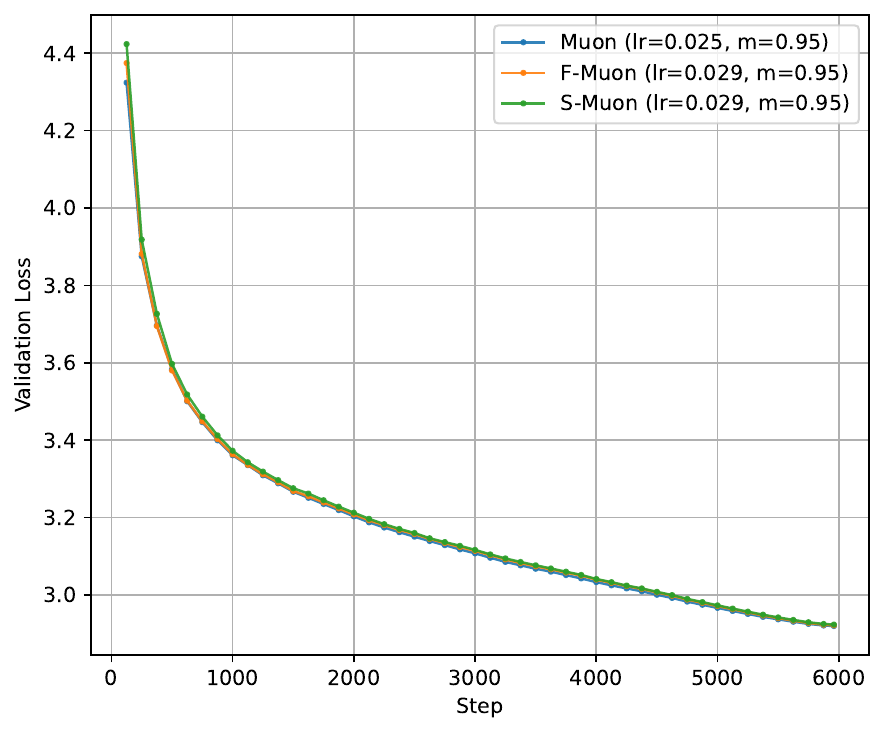}
\caption{Validation loss for GPT-2 Medium.}\label{fig:med_gpt_loss}

   \end{figure}
 
 \subsection{NanoGPT Fine-Tuning}
 \label{subsec:finetune}
 
 We fine-tuned the NanoGPT framework by Karpathy \citep{Karpathy2022} using the standard GPT-2 Medium configuration. We selected the \emph{TinyStories} corpus \citep{eldan2023tinystoriessmalllanguagemodels} as the training dataset for its high entropy and structural diversity, which amplify the differences in optimization dynamics. All training runs were initialized with the same random seed, weight initialization, and learning rate schedule to ensure that performance differences arose solely from the choice of the optimizer.
 
 For all layers with one-dimensional outputs, we used AdamW with a learning rate of $1 \cdot 10^{-3}$. Since momentum exhibited negligible influence on training dynamics, it was held constant across all experiments. \figureref{fig:nanogptftexps} presents a comparative analysis of optimization performance. Notably, F-Muon is significantly more robust to the choice of learning rate than Muon and S-Muon.
 
 \begin{figure}[htbp]
 
\includegraphics[width=0.7\linewidth]{optimizers_comparison.pdf}
\caption{Comparison of validation loss for Muon, F-Muon, and S-Muon across a range of learning rates. Stars denote the best learning rates for each optimizer. $\alpha=1/2$ for F-Muon and S-Muon and \texttt{sign\_sgd\_coeff} = 1/3 for S-Muon.}\label{fig:nanogptftexps}

 \end{figure}
 
 \figureref{fig:nanogptftexps2} presents the training and validation loss curves at the optimal learning rate for each optimizer. While F-Muon and S-Muon maintain consistent behavior, vanilla Muon achieves the lowest loss overall, outperforming other algorithms by a small margin.
 
 \begin{figure}[htbp]
 
\includegraphics[width=\linewidth]{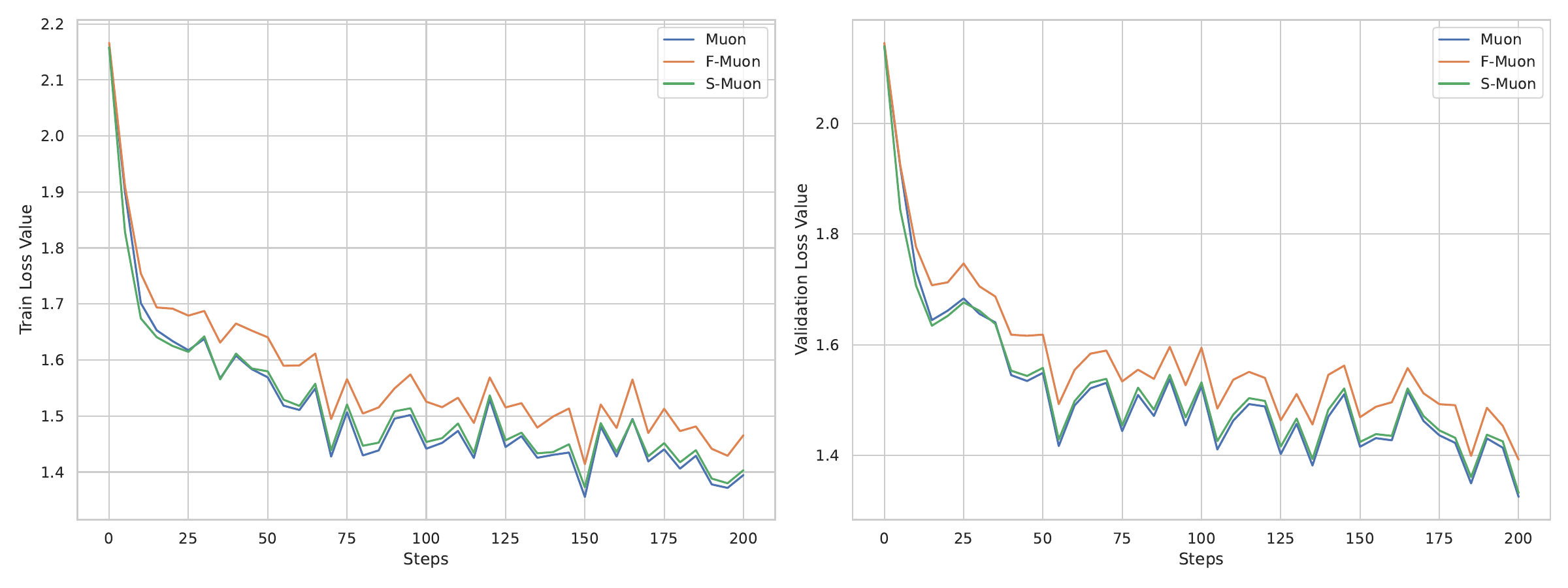}
\caption{Train and validation loss at the optimal learning rate for each optimizer.}\label{fig:nanogptftexps2}

 \end{figure}
 
 \section{Related Work and Discussion}
 
 \subsection{Algorithms for Vectors $\rightarrow$ Algorithms for Matrices}
 
 The updates of LMO optimizers exhibit striking similarities between vector $l_p$ norms and matrix Schatten $S_p$ norms, as illustrated in \tableref{tbl:mat_vs_vec_lmo}. These parallels extend beyond the update formulas themselves to empirical performance characteristics. SignSGD is related to Adam, as noted in \citep{bernstein2024oldoptimizernewnorm}, and both SignSGD and Muon perform well during training large models \citep{zhao2025deconstructing,liu2025muonscalablellmtraining}. NSGD maintains identical updates in both vector and matrix cases. Greedy Coordinate Descent can hardly be applied to high-dimensional problems, which provides a perspective on why one-rank Neon underperforms in such settings.
 
 \begin{table}[htbp]
 
\caption{LMO optimizers in Schatten $S_p$ norms and in $l_p$ norms. $g$ is the gradient. When it is a matrix, $g = \mU \mSigma \mV^\top$.}\label{tbl:mat_vs_vec_lmo}
     \begingroup
     \def\arraystretch{1.2}
     \resizebox{\linewidth}{!}{%
       \begin{tabular}{lccc}
         \arrayrulecolor{black}\toprule
         Algorithm                             & LMO constraint set $\mathcal{D}$ & LMO                                                         & Reference                            \\
         \arrayrulecolor{black}\hline
         \arrayrulecolor{black}\hline
         Normalized SGD                     & $l_2$-ball, $S_2$-ball          & $-\eta \tfrac{g}{\norm{g}_2} = -\eta \tfrac{g}{\normf{g}}$ & \citep{hazan2015beyond}              \\
         Momentum Normalized SGD            & Ball in $l_2$, or Ball in $S_2$ & $-\eta \tfrac{g}{\norm{g}_2} = -\eta\tfrac{g}{\normf{g}}$  & \citep{cutkosky2020momentum}         \\
         \arrayrulecolor{black}\hline
         SignSGD                            & Ball in Max-norm $l_\infty$     & $-\eta \sign(g)$                                            & \citep[Thm. 1]{bernstein2018signsgd} \\
         Signum                             & Ball in Max-norm $l_\infty$     & $-\eta \sign(g)$                                            & \citep[Thm. 3]{bernstein2018signsgd} \\
         \hdashline
         Muon                               & Ball in Spectral $S_\infty$     & $-\eta \mU\mV^\top$                                             & \citep{jordan2024muon}               \\
         \arrayrulecolor{black}\hline
         Gauss-Southwell Coordinate Descent & Ball in $l_1$                   & $-\eta \sum_{i \in \argmax|g_i^t|} \sign(g_i^t)e_i$                       & \citep[p. 19]{shi2016primer}         \\
         \hdashline
         Neon                               & Ball in Nuclear $S_1$           & $-\eta \vu_1 \vv_1^\top$                                        & This work                            \\
         \arrayrulecolor{black}\bottomrule
       \end{tabular}%
     }%
     \endgroup

 \end{table}
 
 \subsection{Theory behind Muon}
 Since Muon \citep{jordan2024muon} is a highly efficient optimizer for functions of weight matrices, substantial research has focused on two objectives: further improving its performance and theoretically explaining its success.
 
 There has been a prolonged gap in the theoretical understanding of Muon, excluding the simplistic derivation of Muon \citep{bernstein2025deriving} that was based on \citep{bernstein2024oldoptimizernewnorm}. This gap, in our view, remains incompletely closed. For example, Kovalev \citep{kovalev2025understanding} provided convergence guarantees for Muon in various settings, from which, however, Muon's empirical supremacy cannot be recovered. Indeed, although the obtained bounds depend on the norm choice, the convergence asymptotics remain the same as for NSGD and other optimizers: $K = \cO(\epsilon^{-4})$ in the $L$-smooth stochastic case.
 
 A similar limitation affects \citep{riabinin2025gluon}, where the $L$-smoothness assumption is replaced with a more practical $(L_0, L_1)$-smoothness. By estimating smoothness, the authors recovered the optimal fine-tuned step sizes reported by \citep{pethick2025training}. However, they have not theoretically demonstrated the optimality of the spectral or RMS-to-RMS norm, which is observed in practice, as our comparison with NSGD and Neon highlights.
 
 A common limitation of these analyses is their focus on convergence measured by the gradient norm. As we showed in our CIFAR experiments (\sectionref{sec:cifar_underhood}), the stochastic gradient norm may decrease by only a factor of ten when the training accuracy reaches 100\%.
 
 Another way to understand Muon is to focus on particular problems, like multi-class linear classification on separable data \citep{fan2025implicit}. Curiously, Neon has the same margin convergence rate as Muon because both are based on Schatten norms. That being said, the norms for margins are different: Muon has an implicit bias toward solutions that minimize the spectral norm of the weight matrix, while Neon tends to minimize its nuclear norm.

 We hypothesize that the stark performance discrepancy between Neon and Muon, both of which are described by the Stochastic Conditional Gradient method \citep{pethick2025training} or Gluon \citep{riabinin2025gluon} frameworks, lies in the structure of the norm ball or in the preconditioner interpretation of the algorithm \citep{pethick2025trainingneuralnetworksscale}, which warrants further investigation.

 \subsection{Improvements of Muon $\rightarrow$ of Fanions, F-Fanions, and S-Fanions as well}
 A large number of applications and improvements of Muon have been proposed in less than a year. Liu~et~al.~\cite{liu2025muonscalablellmtraining} adapted the algorithm for training language models larger than NanoGPT. Shah~et~al.~\cite{shah2025practical} enabled efficient hyperparameter transfer by combining Muon with maximal update parametrization. To construct their COSMOS optimizer, Chen~et~al.~\cite{chen2025cosmoshybridadaptive} applied computationally intensive updates of the SOAP optimizer \citep{vyas2025soap} to a low-dimensional ``leading eigensubspace'' while using memory-efficient methods like Muon for the remaining parameters. Amsel~et~al.~\cite{amsel2025polar} and Grishina~et~al.~\cite{grishina2025chebyshev} proposed more efficient alternatives to tuning the coefficients in the Newton-Schulz algorithm. Si~et~al.~\cite{si2025adamuon}, Veprikov~et~al.~\cite{veprikov2025preconditioned}, and Li~et~al.~\cite{li2025normuon} introduced AdaMuon, MuAdam, and NorMuon, respectively, which combine gradient orthogonalization with adaptivity.

 Fanios, F-Fanions, and S-Fanions benefit from these improvements and many others. They could be transformed into Drop-Fanions by updating only the selected layers, as in \citep{gruntkowska2025dropmuon}. They can be viewed as approximations of the Non-Euclidean Proximal Point Method for the corresponding norms \citep{gruntkowska2025noneuclideanbroximal}. They can be clipped to produce ClippedScion-like algorithms \citep{pethick2025generalizedgradientnormclipping}. They could be made more memory-efficient through the zero-order techniques that proved effective for Muon \citep{petrov2025zeromuon}. Finally, the results from \citep{shulgin2025inexactmuon} can be used to explain the robustness of Muon to the norm changes observed in our experiments and to theoretically derive faster yet effective approximate schemes for calculating the LMO; power iterations with a limited number of iterations represent a promising direction for this analysis.
 
 \subsection{The Nuclear Norm in the LMO}
 
 We discovered during the preparation of this article that the nuclear norm has already been explored in the context of the linear minimization oracle. Pethick et al. \citep{pethick2025sam} applied it to create $\nu$SAM, a novel sharpness-aware minimization technique. It would be interesting to substitute $\normn{\cdot}$ with $\normkfk{\cdot}^\dagger$ in their approach. Since the SAM neighborhood becomes more diverse, using $k > 1$ might enhance the accuracy boost from using SAM while preserving a small memory footprint and minimal time overhead if Dion-style power iterations are employed.
 
 \subsection{The LMO and Error Feedback}
 
 As previously mentioned, rank-$k$ unsharded Dion without error feedback and update scaling is equivalent to Fanion-$k$. Since error feedback is crucial for Dion, as demonstrated by the ablation study in \citep{ahn2025dion}, F-Fanions and S-Fanions would benefit from it as well. In federated learning, error feedback proves effective even for compressed Muon \citep{gruntkowska2025efmuon}. Fanions and S-Fanions offer a transmission advantage, requiring fewer bits: $\sum_{i=1}^k \vu_i \vv_i^\top$ can be efficiently transmitted as $\{\vu_1, \vv_1, \dots, \vu_k, \vv_k\}$ ($(m+n)\times k$ floats), while the sign component can be encoded in $m \times n$ bits. Thus, compression is inherently built into the representation. Moreover, there is an intriguing possibility to construct differentially private Fanions and S-Fanions using more optimal non-Gaussian noise, as was done with DP-SignSGD \citep{jang2024dplogsign}. We leave this for future research.

 \section{Conclusions}

 In this article, we addressed the central question of whether one should constrain by the spectral or any other operator norm in deriving Muon-like updates, and how alternative norms affect performance and computational cost. Our answer is that the choice of the matrix norm is remarkably flexible: properly tuned variants based on alternative norms can match or even slightly exceed Muon's performance on real-world tasks, while offering additional benefits such as improved learning rate robustness.

 We generalized several successful algorithms to linear minimization oracle (LMO) based algorithms using the family of norms dual to Ky Fan $k$-norms, yielding the Fanion family with low-rank updates. We further proposed the technique of combining them with Normalized SGD or SignSGD, creating the F-Fanion and S-Fanion families. Our experiments demonstrate that F-Muon and S-Muon achieve competitive performance with Muon on CIFAR-10 airbench and NanoGPT speedrun tasks, confirming that the underlying norm constraint can be significantly modified without sacrificing effectiveness.

 However, our results also reveal that not every LMO-based algorithm is effective: Neon (rank-one Fanion) underperforms despite sharing the same theoretical convergence asymptotics as Muon in existing bounds. We suggest that future work on non-Euclidean LMO algorithms should explain the observed empirical superiority of Muon over other Fanions in the corollaries to the convergence theorems, and ideally account for the robustness of F-Muon and S-Muon as well.

\section*{Declarations}

\noindent\textbf{Funding} The authors did not receive support from any organization for the submitted work.

\noindent\textbf{Competing Interests} The authors have no relevant financial or non-financial interests to disclose.

\noindent\textbf{Data Availability} No datasets were generated during the current study. All data used in the experiments is publicly available.

\noindent\textbf{Author Contributions} IO suggested using the nuclear norm in the Bernstein--Newhouse framework~\cite{bernstein2024oldoptimizernewnorm}. DM presented the problem at the MIPT optimization course, supervised the project, and helped to revise the manuscript. IK suggested using composite norms (though not the ones that induce Fanions, F-Fanions, or S-Fanions) and helped to draft and revise the manuscript. NK suggested the Lanczos algorithm as the most precise means to compute Fanions' updates, conducted experiments to prove this, and wrote \sectionref{sec:matrix-side}. AV conducted the finetuning of NanoGPT on TinyStories and wrote \sectionref{subsec:finetune}. All other work was done by AK: constructing Fanions, F-Fanions, and S-Fanions, conducting experiments on the smooth convex problem and on the CIFAR-10 airbench, pretraining NanoGPT and GPT-2 Medium, and writing the manuscript.

\begin{acknowledgements}
We thank Dmitry Kovalev and Egor Shulgin for helpful discussions.
\end{acknowledgements}

\bibliographystyle{spmpsci}
\bibliography{icomp2024_conference}
 
 \appendix
 \section{LMO for Neural Networks}\label{sec:lmo_for_nn}
 
 In a typical neural network, the objective function $F$ depends on a set of weight matrices $\{ \mW_1, \mW_2, \ldots \}$. The optimization framework we have described is applied in a layer-wise fashion. At each iteration $t$, a stochastic gradient $g(\mX^t, \xi^t)$ is computed using a mini-batch of data with the $\xi^t$ noise via backpropagation. This yields a separate gradient component, $\mG_i^t$, for each matrix $\mX_i$. The LMO-based update rule is then applied to each matrix $\mX_i$ using its corresponding gradient component $\mG_i^t$.
 
 \section{Proof of Lemma on the Dual to Convex Combinations}\label{sec:proof_dual_conv_comb}
 
 We provide here the proof of \lemmaref{lemma:dual_to_conv_comb}.
 
 \begin{proof}
   Let us first prove the lemma for the case $n = 2$. Denote $f(x) = \alpha_1 \norm{x}_{(1)}$ and $g(x) = \alpha_2 \norm{x}_{(2)}$, such that $\norm{x} = f(x) + g(x)$. Recall two standard facts:
   \begin{enumerate}
     \item For any norm $\norm{\cdot}$ and $\lambda>0$,
           \[
             (\lambda \norm{\cdot})^*(y) =
             \sup_{x}\bigl( \langle y,x \rangle - \lambda \norm{x}\bigr)
             = \delta_{\lambda \cB_{\norm{\cdot}^\dagger}}(y)\,,
           \]
           i.e., the indicator function of the scaled dual ball.
     \item The Fenchel conjugate of a sum satisfies
           \[
             (f+g)^*(y) = \inf_{u+v=y} \bigl(f^*(u) + g^*(v)\bigr)\,.
           \]
   \end{enumerate}
   Applying these to $f$ and $g$, we have
   \[
     f^*(u) = \delta_{\alpha_1 \cB_{\norm{\cdot}_{(1)}^\dagger}}(u)\,,
     \quad
     g^*(v) = \delta_{\alpha_2 \cB_{\norm{\cdot}_{(2)}^\dagger}}(v)\,.
   \]
   Thus,
   \[
     \norm{\cdot}^*(y)
     = (f+g)^*(y)
     = \inf_{u+v=y}
     \bigl(
     \delta_{\alpha_1 \cB_{\norm{\cdot}_{(1)}^\dagger}}(u)
     +
     \delta_{\alpha_2 \cB_{\norm{\cdot}_{(2)}^\dagger}}(v)
     \bigr)
     =
     \delta_{\alpha_1 \cB_{\norm{\cdot}_{(1)}^\dagger}+\alpha_2 \cB_{\norm{\cdot}_{(2)}^\dagger}}(y).
   \]
   By definition, the conjugate of a norm is exactly the indicator of its dual unit ball:
   \[
   \norm{\cdot}^*(y) = \delta_{\cB_{\norm{\cdot}^\dagger}}(y)\,.
   \]
   Therefore, $\cB_{\norm{\cdot}^\dagger} = \alpha_1 \cB_{\norm{\cdot}_{(1)}^\dagger} + \alpha_2 \cB_{\norm{\cdot}_{(2)}^\dagger}$.
 
   Now we prove the general case by induction. The base case ($n=2$) is already proven. Suppose that the assumption of the lemma holds for $n = k$. Then, for $n = k + 1$,
   \[
   \norm{x} = \sum_{i = 1}^k \alpha_i \norm{x}_{(i)} + \alpha_{k + 1}\norm{x}_{(k + 1)} = \norm{x}_{(1:k)} + \alpha_{k + 1}\norm{x}_{(k + 1)}\,.
   \]
   Applying the result for $n=2$ combined with the induction assumption, we obtain
   \[
   \cB_{\norm{\cdot}^\dagger} = \cB_{\norm{\cdot}_{(1:k)}^\dagger} + \alpha_{k + 1} \cB_{\norm{\cdot}_{(k + 1)}^\dagger} = \sum_{i = 1}^{k + 1} \alpha_i \cB_{\norm{\cdot}_{(i)}^\dagger}\,,
   \]
   which proves the lemma.\qed
\end{proof}
 
 \section[Norms for F-Muon and F-Neon]%
         {Norms $\normfstar{\cdot}^\dagger$ and $\normftwo{\cdot}^\dagger$}
 \label{sec:f2_fstar_norms}
 
 Based on \lemmaref{lemma:dual_to_conv_comb}, we immediately find $\normfstar{\cdot}^\dagger$, which is related to the F-Muon update. Indeed, after setting $\beta=1-\alpha$ and remembering that for smooth and bounded cases we can use $\min$ instead of $\inf$, we obtain
 \begin{equation}\label{eq:normfstardual}
   \normfstar{\mY}^\dagger = \min_{\mZ} \min_t\{t, s.t. \norms{\mZ}\leq \alpha t, \normf{\mY-\mZ}\leq (1-\alpha) t\}\,.
 \end{equation}
 
 If $\alpha = 1$, then $\mZ = \mY$, and we get $\normfstar{\mY}^\dagger = \norms{\mY}$. If $\alpha = 0$, then $\mZ = 0$, and we get $\normfstar{\mY}^\dagger = \normf{\mY}$.
 
 Similarly, we find $\normftwo{\cdot}^\dagger$, which is related to the F-Neon update:
 \begin{equation}\label{eq:normftwodual}
   \normftwo{\mY}^\dagger = \min_{\mZ} \min_t\{t, s.t. \normn{\mZ}\leq \alpha t, \normf{\mY-\mZ}\leq (1-\alpha) t\}\,.
 \end{equation}
 
 If $\alpha = 1$, then $\mZ = \mY$, and we get $\normftwo{\mY}^\dagger = \normn{\mY}$. If $\alpha = 0$, then $\mZ = 0$, and we get $\normftwo{\mY}^\dagger = \normf{\mY}$.
 
 Unfortunately, $\normfstar{\cdot}^\dagger$ and $\normftwo{\cdot}^\dagger$ do not have simple closed-form expressions and cannot be computed as easily as their duals.
 
 \section{Visualization of Different Matrix Norms}
 \subsection{Duals to F* and F2 Norms}
 
 It follows from \lemmaref{lemma:dual_to_conv_comb} that the norm ball in $\normfstar{\cdot}^\dagger$ is the Minkowski sum of the norm ball in $\alpha\normn{\cdot}$ and $(1-\alpha)\normf{\cdot}$, and the norm ball in $\normftwo{\cdot}^\dagger$ is the Minkowski sum of the norm ball in $\alpha\norms{\cdot}$ and $(1-\alpha)\normf{\cdot}$.
 
 In \figureref{fig:fduals}, we plot these norms. The x-axis and y-axis represent the singular values $\sigma_1$ and $\sigma_2$, respectively, of a matrix from $\R^{m\times n}$ with $\min\{m,n\}=2$.
 
 \begin{figure}[htbp]

     \subfigure[LMO balls for F-Muon for different $\alpha$]{\label{fig:fstardual}%
       \includegraphics[width=0.48\linewidth]{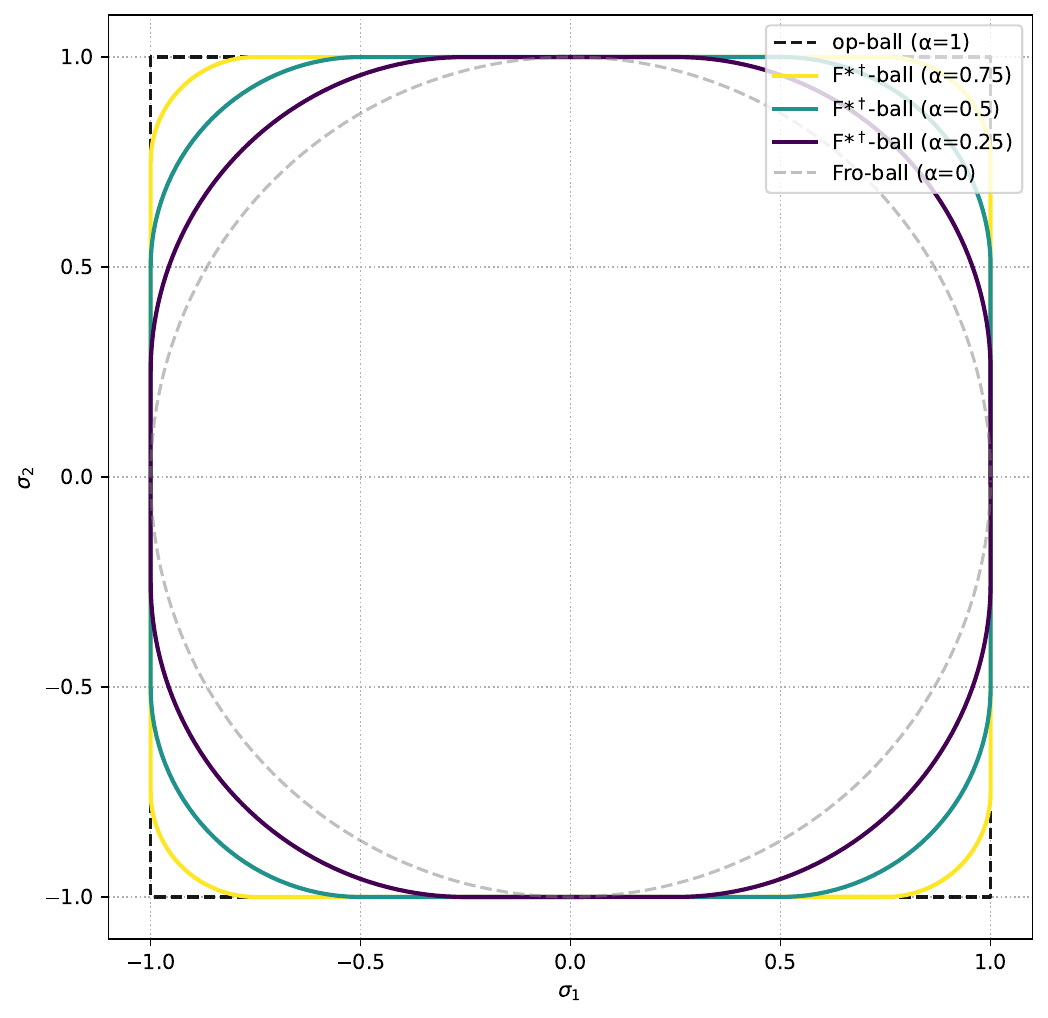}}%
     \hfill
     \subfigure[LMO balls for F-Neon for different $\alpha$]{\label{fig:ftwodual}%
       \includegraphics[width=0.48\linewidth]{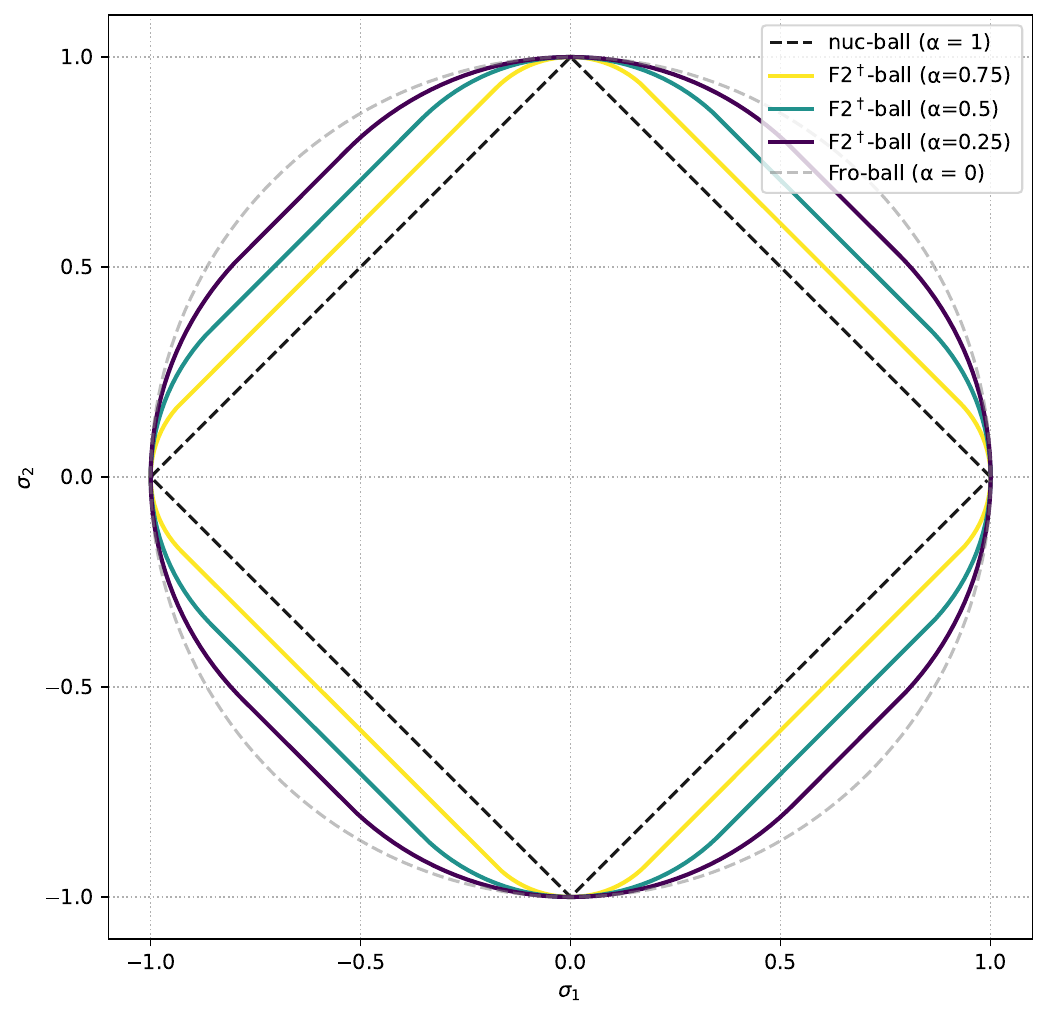}}
   
\caption{Balls in the duals to F* and F2 norms for different $\alpha$.}\label{fig:fduals}

 \end{figure}
 
 \subsection{A Ky Fan Norm and Its Dual}\label{subsec:kyfan2norm}
 
 While 1-balls in $l_\infty$, $l_1$, and $l_2$ norms are well-known from textbooks, the Ky Fan $k$-norm presents a more intricate structure.
 
 To showcase the complex geometry of the Ky Fan $k$-norm and its dual, we present \figureref{fig:kyfan_combined}, which displays the unit ball in the Ky Fan $2$-norm (\figureref{fig:kyfan}) and its dual (\figureref{fig:kyfandual}). The x-, y-, and z-axes represent the singular values $\sigma_1$, $\sigma_2$, and $\sigma_3$, respectively, of a matrix from $\R^{m\times n}$ with $\min\{m,n\}=3$. In this representation, we do not sort the singular values. We plot unit balls in the Top-$2$ norm $\max\{\abs{x}+\abs{y}, \abs{x}+\abs{z}, \abs{y}+\abs{z}\}$ and its dual norm $\max\{\max(\abs{x},\abs{y},\abs{z}), (\abs{x}+\abs{y}+\abs{z})/2\}$. The resulting balls exhibit significantly greater complexity than those in $l_\infty$, $l_1$, and $l_2$ norms.
 
 \begin{figure}[htbp]
  \centering
  \subfigure[The Ky Fan $2$-norm]{\label{fig:kyfan}%
      \includegraphics[width=0.45\linewidth]{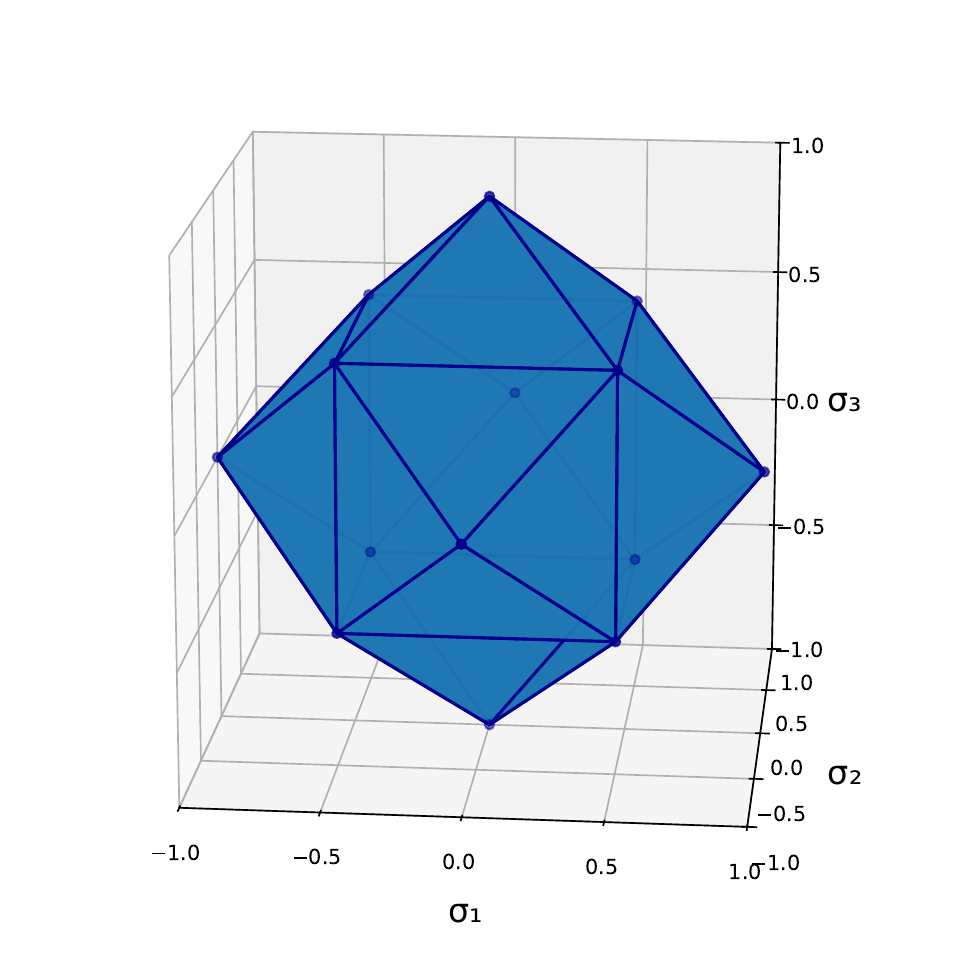}}%
  \hfill
  \subfigure[The dual of the Ky Fan $2$-norm]{\label{fig:kyfandual}%
      \includegraphics[width=0.45\linewidth]{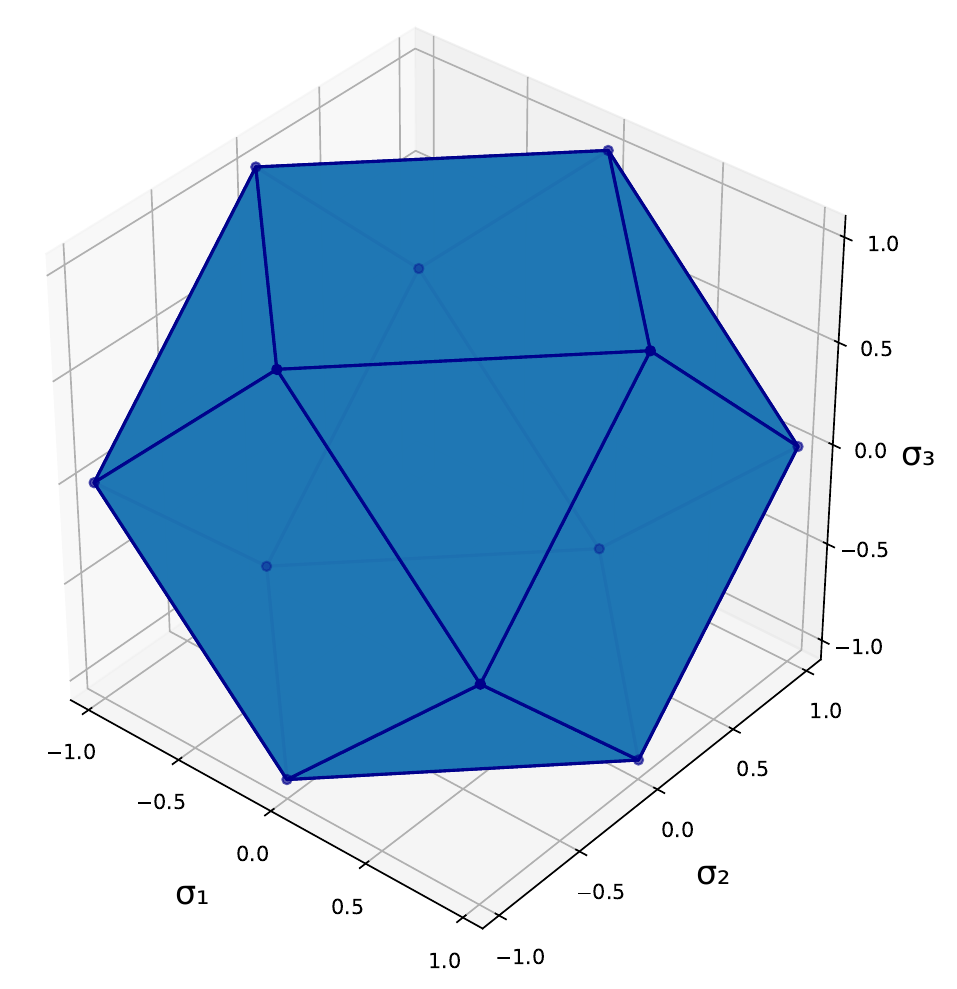}}

  \caption{The Ky Fan $2$-norm and its dual in the space of three singular values.}
  \label{fig:kyfan_combined}
\end{figure}
 
 These balls can be described more simply using the results from Yu~\cite{yu2012arithmetic}. The Ky Fan $2$-norm ball is an intersection of three $l_1$ balls in $(x,y)$, $(x, z)$, and $(y,z)$ spaces. The unit ball in the dual Ky Fan $2$-norm is an intersection of the unit ball in the $l_\infty$ norm and the $1/2$-ball in the $l_1$ norm.
 
 \section{More Details for the Smooth Convex Problem}
 \label{sec:lls_plots_section}
 The spectral and nuclear norms of the full gradients over iterations, as well as the loss over time, are shown in \figureref{fig:app_lls}. The poor performance of Fanions and F-Fanions in terms of speed for large $k$ is likely caused by an inefficient implementation of TRLan.
 
 \begin{figure}[htbp]

     \subfigure[Spectral norm of the full gradient]{\label{fig:lls_op_grad_norm}%
       \includegraphics[width=0.49\linewidth]{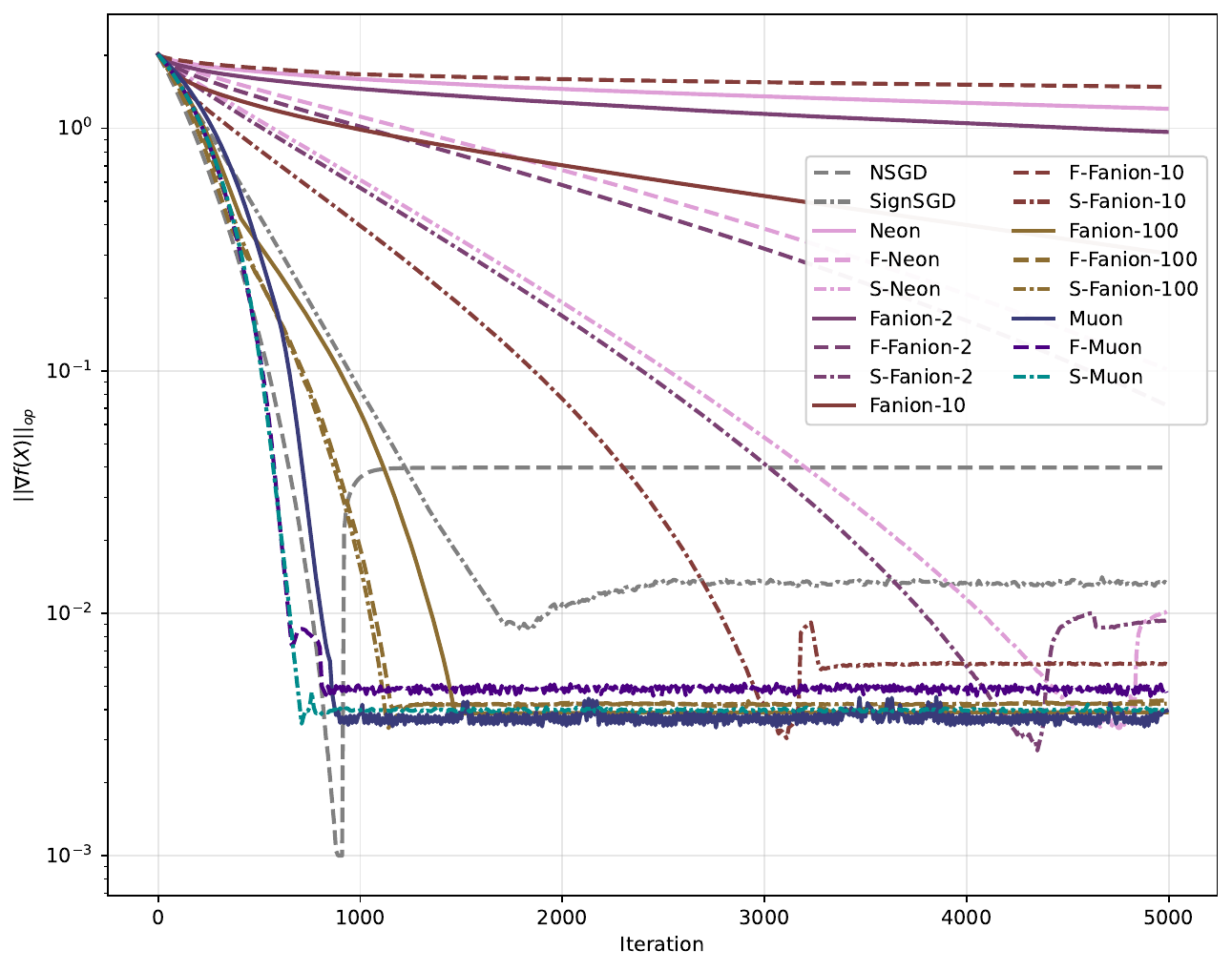}}%
     \hfill
     \subfigure[Nuclear norm of the full gradient]{\label{fig:lls_nuclear_grad_norm}%
       \includegraphics[width=0.49\linewidth]{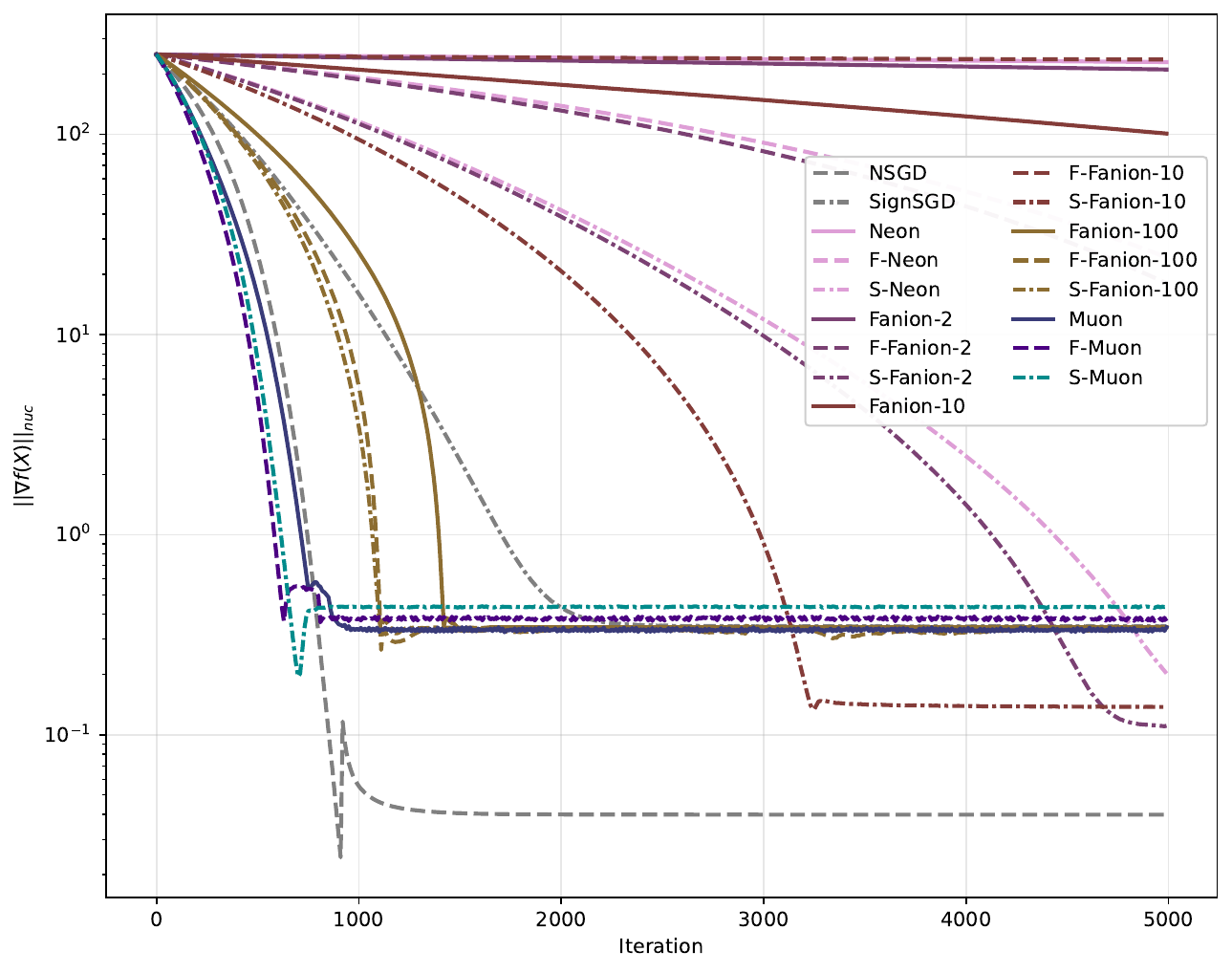}}%
     \vspace{1em}
     \subfigure[Loss over time]{\label{fig:lls_loss_time}%
       \includegraphics[width=0.49\linewidth]{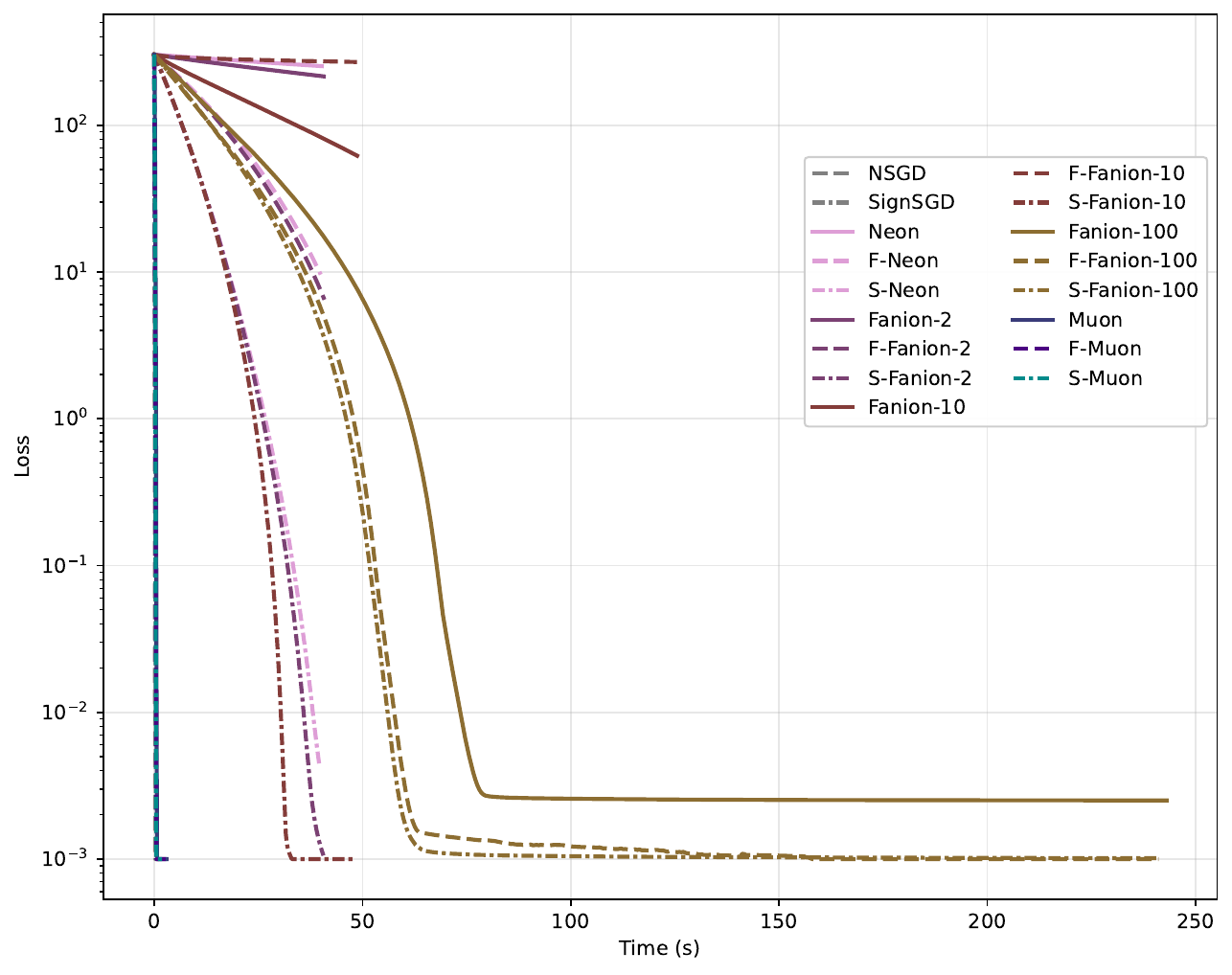}}
   
\caption{More plots for the smooth convex problem (\ref{eq:lls}) for a 500x500 matrix.}\label{fig:app_lls}

 \end{figure}
 
 \section{Gradient Norm Analysis for CIFAR-10 Experiments}
 \label{sec:cifar_underhood}
 
 Theoretical convergence bounds for Muon and other LMO-based algorithms are typically expressed in terms of gradient norms \citep{li2025noteconvergencemuon, kovalev2025understanding, pethick2025training, riabinin2025gluon,kovalev2025noneuclideansgdstructuredoptimization}. Accordingly, we measure these norms during training on a real deep learning problem to validate theoretical predictions.
 
 \begin{table}[htbp]
 
\caption{Parameters for CIFAR-10 airbench. \texttt{sign\_lr\_mult} for S-Muon is 0.003.}\label{tbl:cifar_lrs}
     \begin{tabular}{cccc}
       \arrayrulecolor{black}\toprule
       Method & \texttt{lr}  & \texttt{momentum}  & val\_accuracy, \% \\
       \arrayrulecolor{black}\hline
       NSGD & 0.5 & 0.95 & $91.6 \pm 0.52$ \\
       \hdashline
       SignSGD & 0.003 & 0.95 & $91.54 \pm 0.26$ \\
       \hdashline
       Muon & 0.24 & 0.6 & $94.01 \pm 0.10$ \\
       F-Muon & 0.40 & 0.6 & $94.01 \pm 0.13$\\
       \hdashline
       S-Muon & 0.42 & 0.63 & $94.03 \pm 0.13$ \\
       \hdashline
       Neon & 0.24 & 0.6 & $69.8 \pm 0.5$ \\
       F-Neon & 0.40 & 0.6 & $87.15 \pm 0.24$\\
       \hdashline
       Fanion-5 & 0.24 & 0.6 & $80.69 \pm 1.25$ \\
       F-Fanion-5 & 0.40 & 0.6 & $86.66 \pm 0.65$\\
       \arrayrulecolor{black}\bottomrule
     \end{tabular}%

 \end{table}
 
 We compare Muon, F-Muon, S-Muon, NSGD, SignSGD, Neon, F-Neon, Fanion-5, and F-Fanion-5 on CIFAR-10 airbench. NSGD and SignSGD are not heavily tuned. The hyperparameters of Neon, F-Neon, Fanion-5, and F-Fanion-5 are taken from Muon and F-Muon (see \tableref{tbl:cifar_lrs}). The validation accuracies reported after 8 epochs correspond to the airbench variant with weight normalization applied at each step. However, to remain faithful to the conditions assumed in convergence theorems \citep{kovalev2025understanding,riabinin2025gluon}, we do not normalize network weights later when logging gradient norms.
 
 We log train and validation accuracies (\figureref{fig:cifar_accs}), as well as the Frobenius, spectral, and nuclear norms of the gradients for \texttt{conv1.weight} and \texttt{conv2.weight} across all layers (see \figureref{fig:norms_conv1_vs_conv2} for norms in layer 2, and \figureref{fig:total_norms_cifar} for total gradient norms). During training, gradient norms decrease by at most one order of magnitude, while training accuracy reaches 100\% for Muon, S-Muon, and F-Muon. This observation suggests that convergence bounds of the form $\max\{\frac{A}{\epsilon}, \frac{B}{\epsilon^2}, \frac{C}{\epsilon^3}, \frac{D}{\epsilon^4}\}$ (such as Corollary 2 from Kovalev~\cite{kovalev2025understanding}) should not be simplified to $\frac{D}{\epsilon^4}$ when selecting an optimal algorithm for practical applications: final $\epsilon$ may be quite large.
 
 \begin{figure}[htbp]

     \subfigure[Training accuracy]{\label{fig:train_acc_cifar}%
       \includegraphics[width=0.7\linewidth]{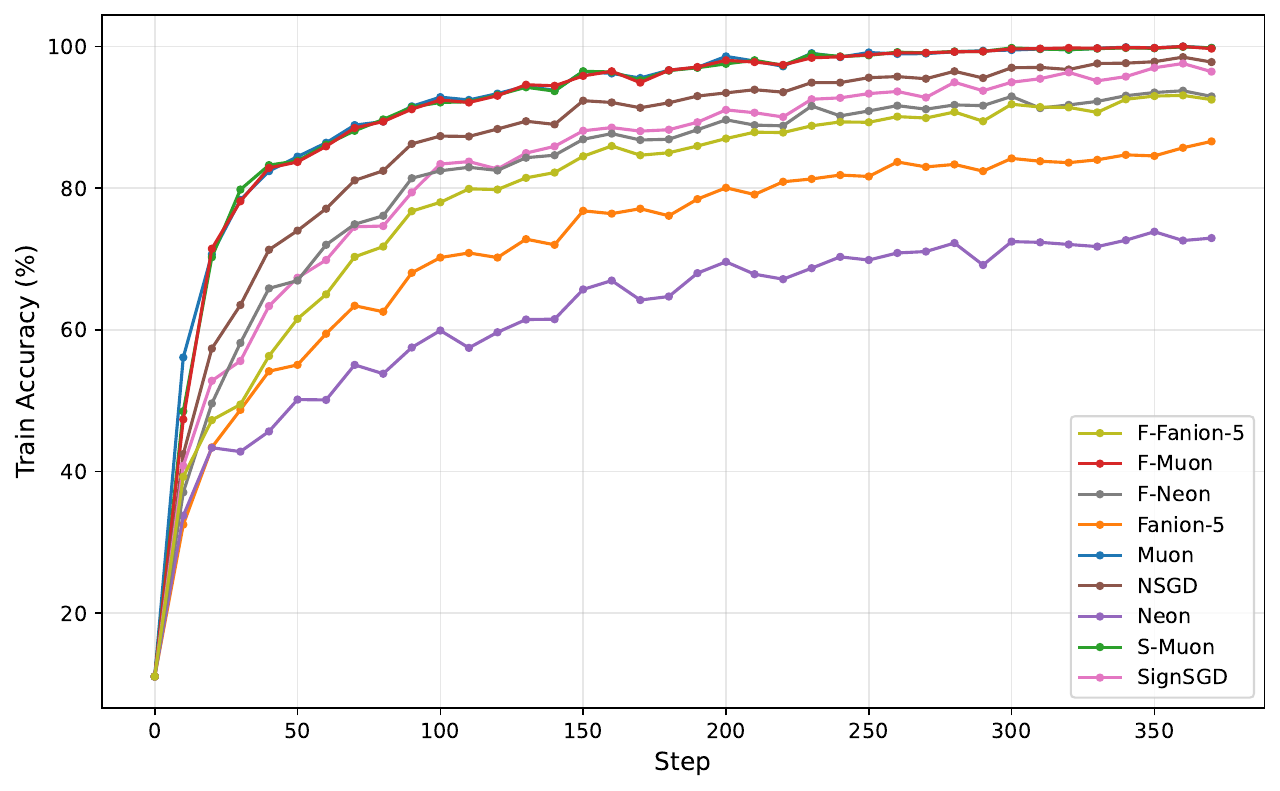}}%
     \vspace{1em}
     \subfigure[Validation accuracy]{\label{fig:val_acc_cifar}%
       \includegraphics[width=0.7\linewidth]{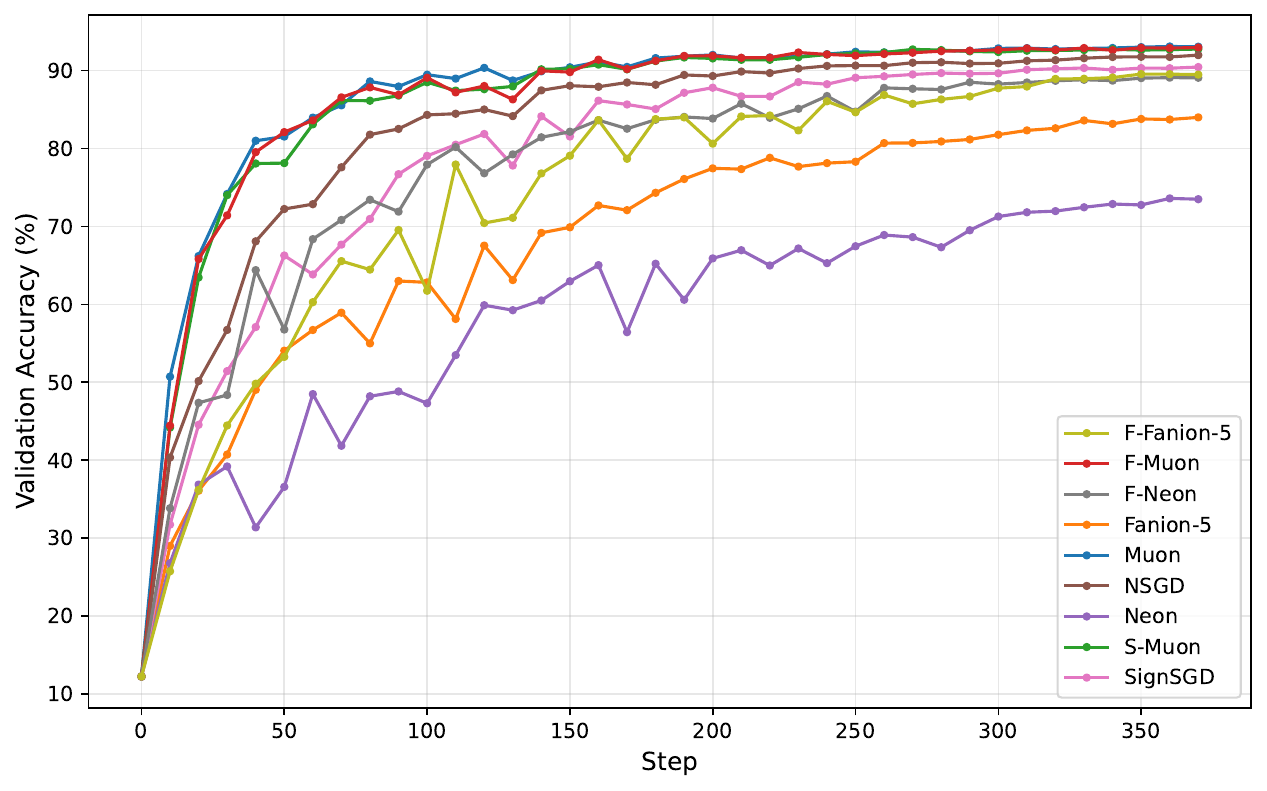}}
   
\caption{Different optimizers on CIFAR-10 airbench without weight normalization.}\label{fig:cifar_accs}

 \end{figure}

 \begin{figure}[htbp]

     \subfigure[Frobenius norm of layer2.conv1]{\label{fig:fro_conv1}%
       \includegraphics[width=0.45\linewidth]{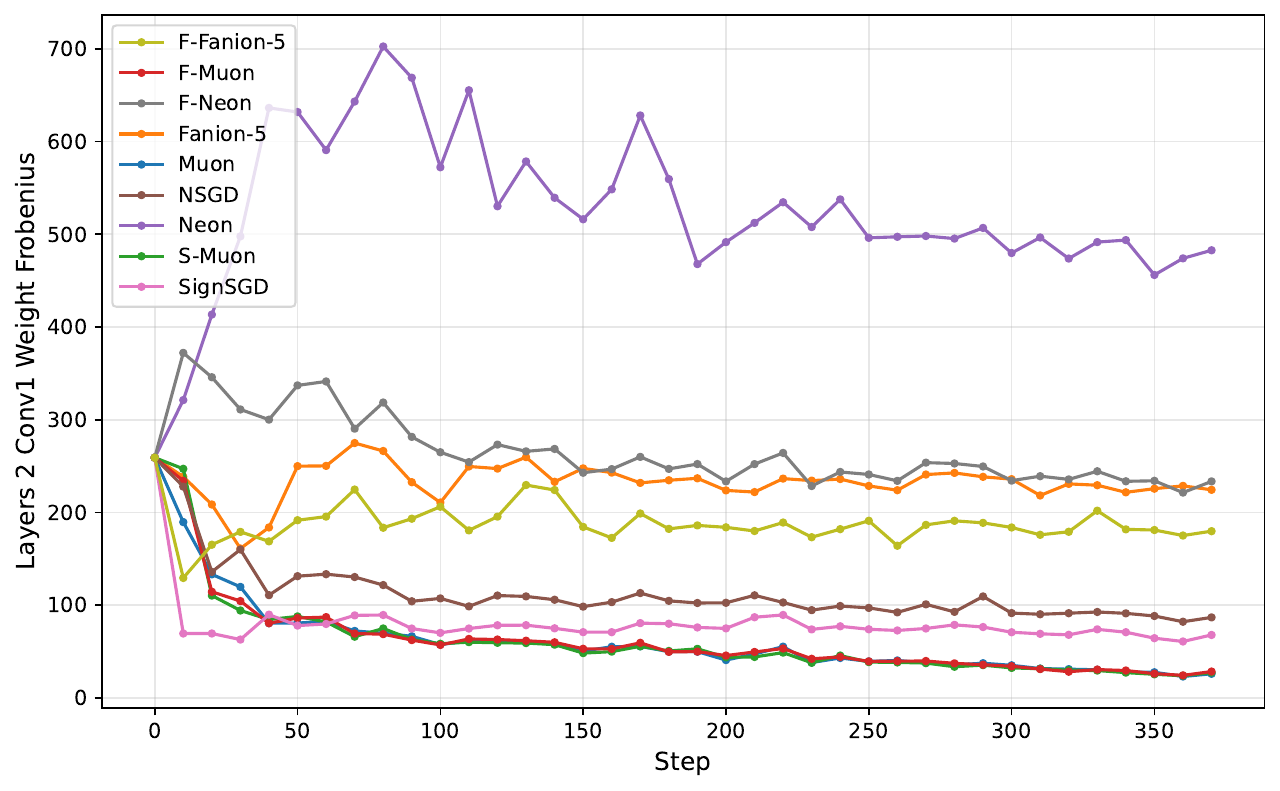}}%
     \hfill
     \subfigure[Frobenius norm of layer2.conv2]{\label{fig:fro_conv2}%
       \includegraphics[width=0.45\linewidth]{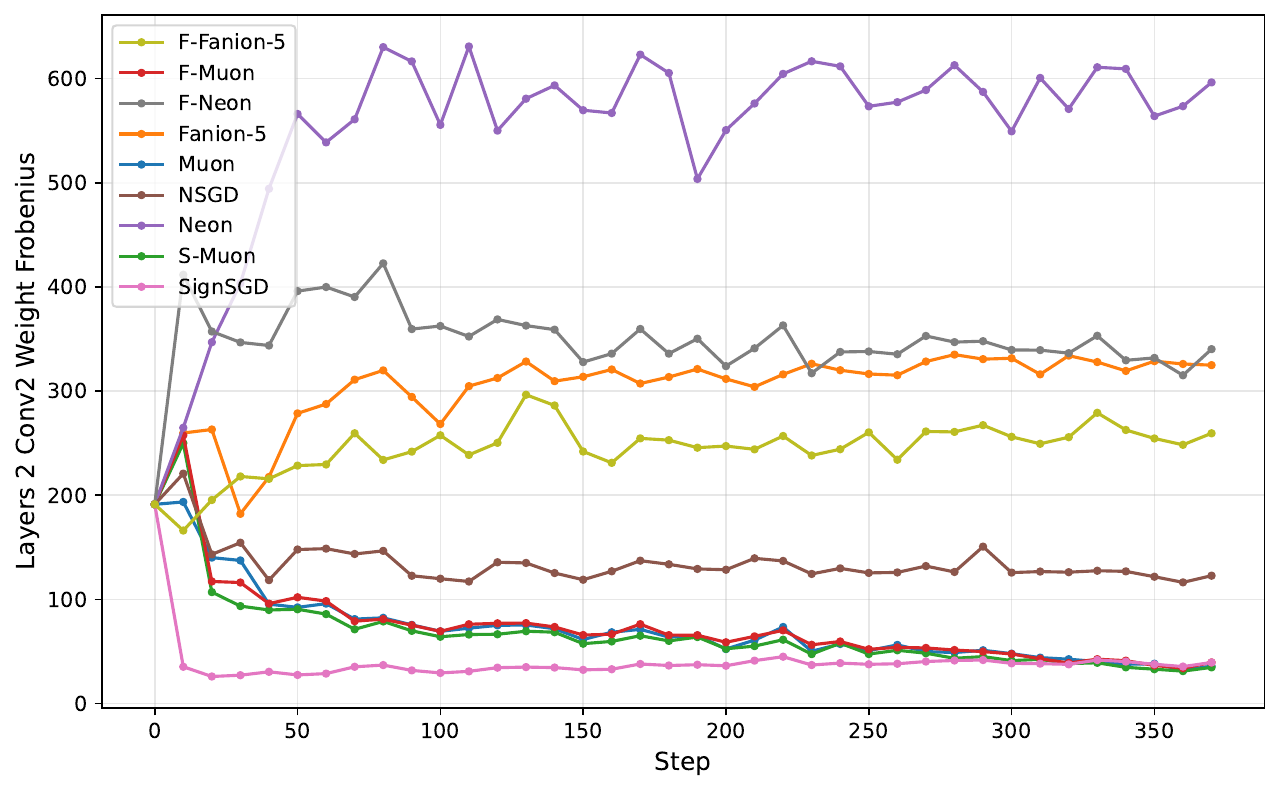}}%
     \vspace{1em}
     \subfigure[Spectral norm of layer2.conv1]{\label{fig:spec_conv1}%
       \includegraphics[width=0.45\linewidth]{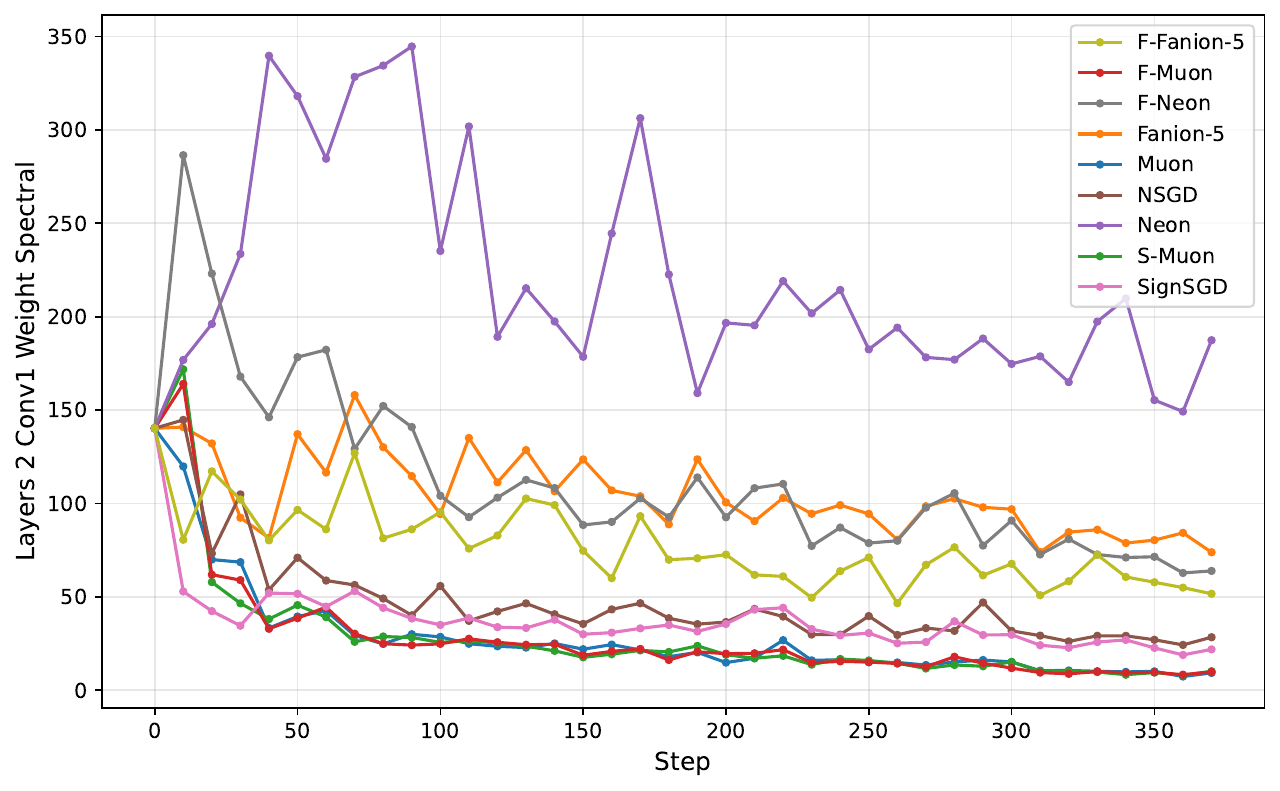}}%
     \hfill
     \subfigure[Spectral norm of layer2.conv2]{\label{fig:spec_conv2}%
       \includegraphics[width=0.45\linewidth]{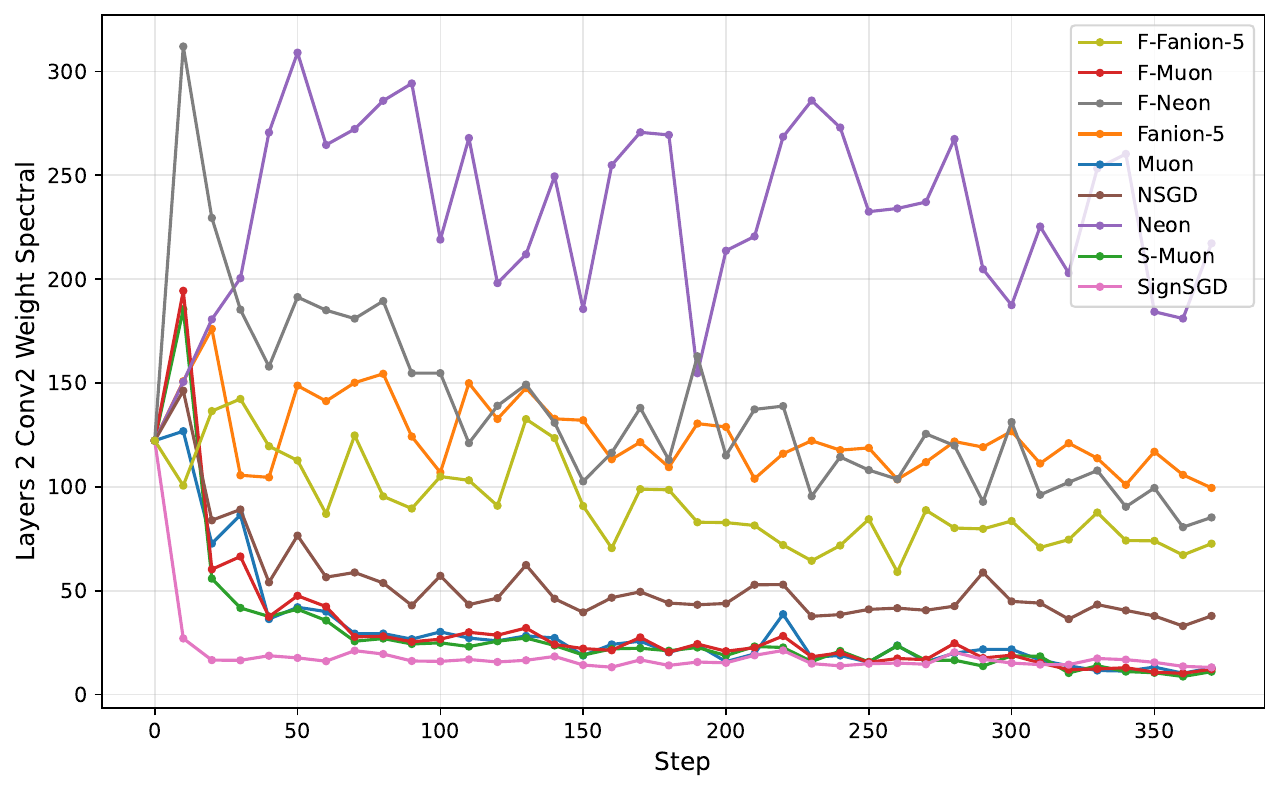}}%
     \vspace{1em}
     \subfigure[Nuclear norm of layer2.conv1]{\label{fig:nuc_conv1}%
       \includegraphics[width=0.45\linewidth]{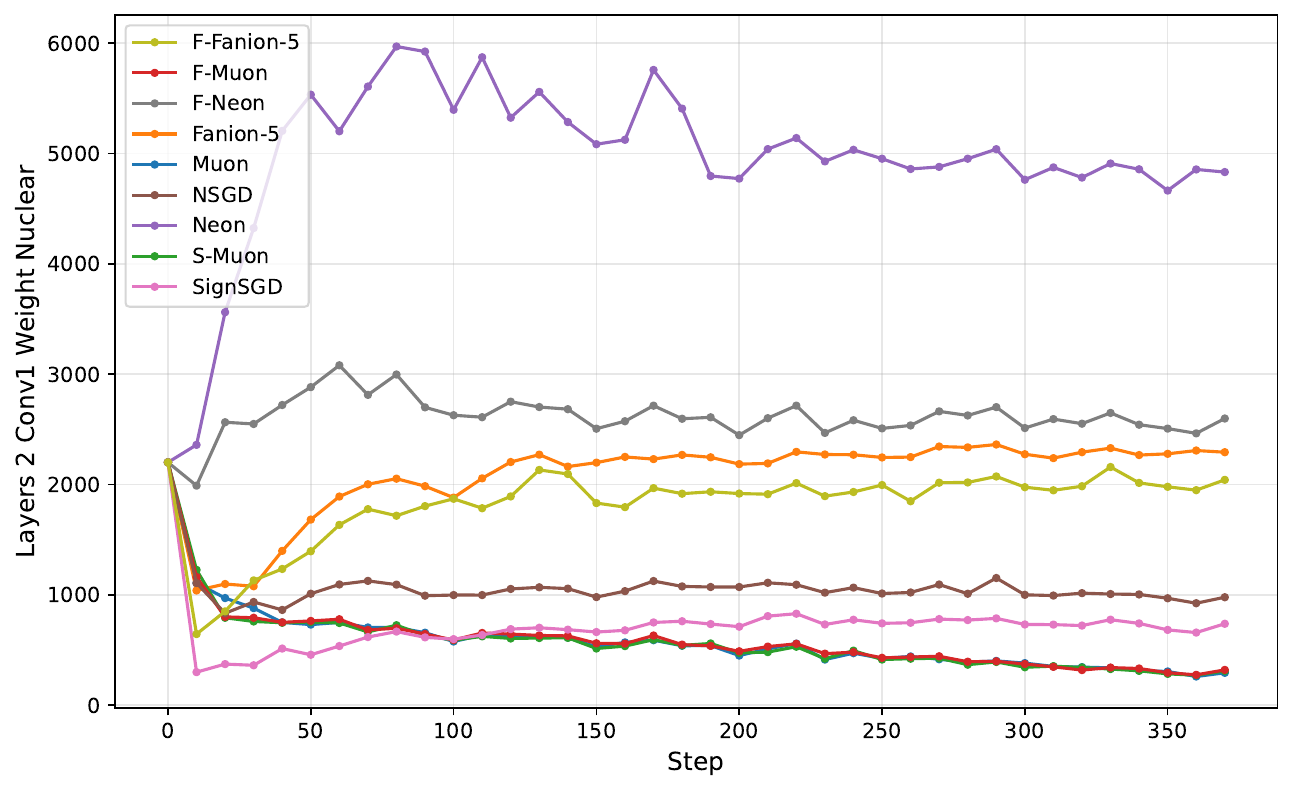}}%
     \hfill
     \subfigure[Nuclear norm of layer2.conv2]{\label{fig:nuc_conv2}%
       \includegraphics[width=0.45\linewidth]{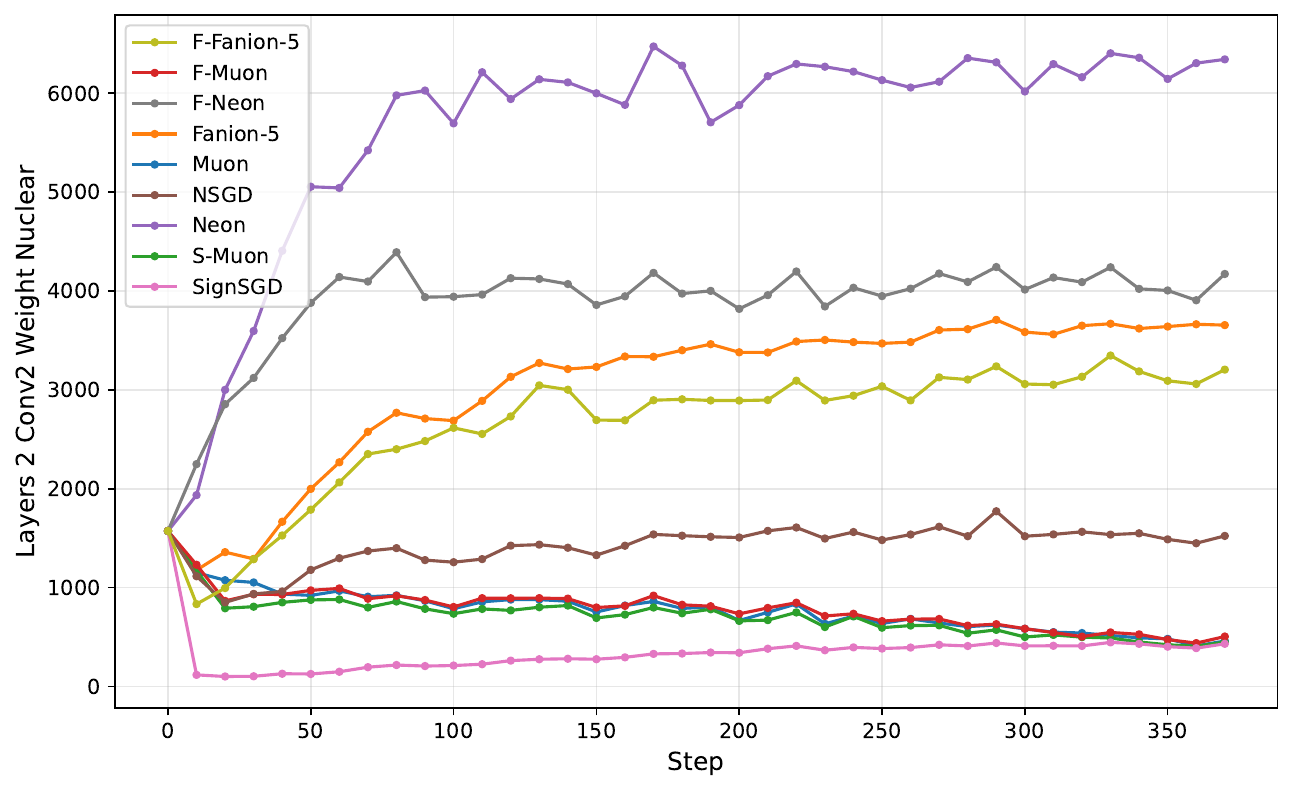}}
   
\caption{Matrix norms of layer2.conv1 (left) and layer2.conv2 (right) gradients of CIFAR CNN.}\label{fig:norms_conv1_vs_conv2}

 \end{figure}
 
 \begin{figure}[htbp]

       \subfigure[Total Frobenius norm]{\label{fig:total_fro_norm_cifar}%
         \includegraphics[width=0.7\linewidth]{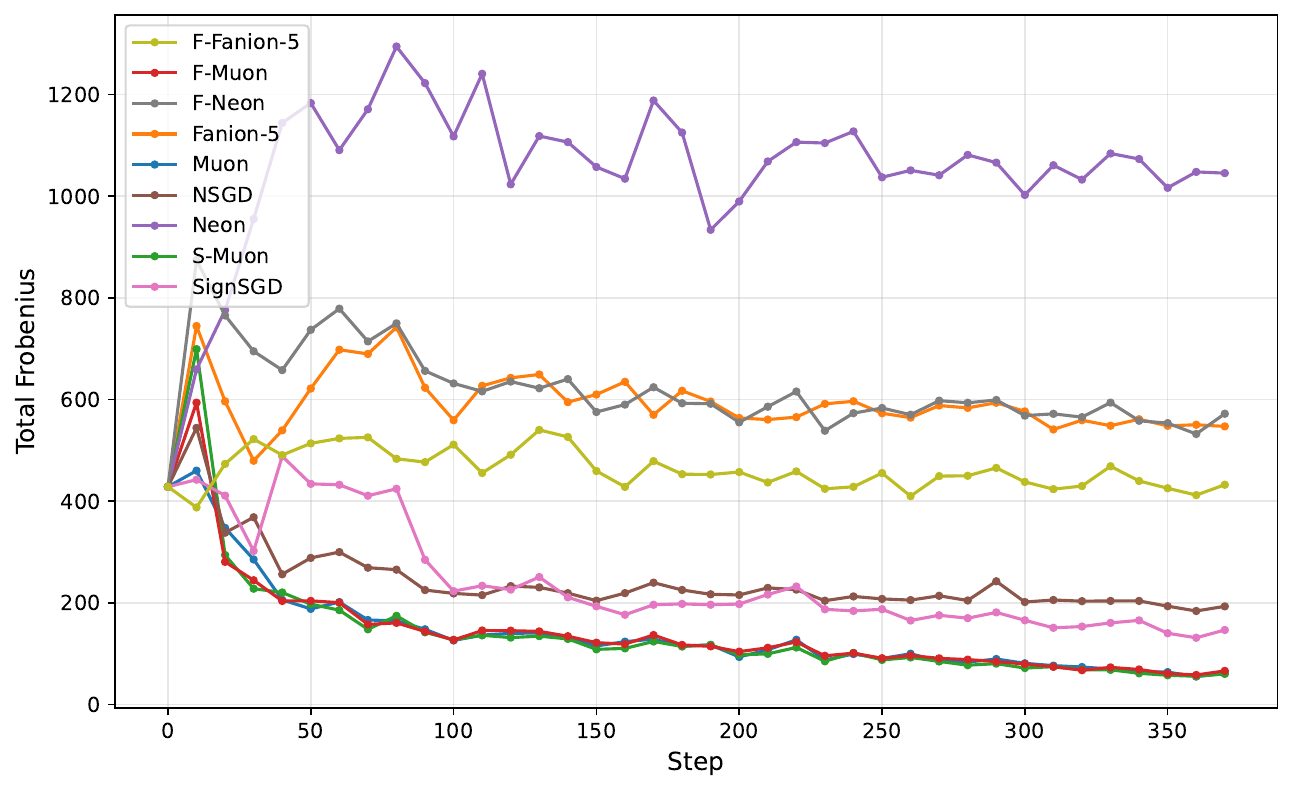}}%
       \vspace{1em}
       \subfigure[Total spectral norm]{\label{fig:total_spectral_norm_cifar}%
         \includegraphics[width=0.7\linewidth]{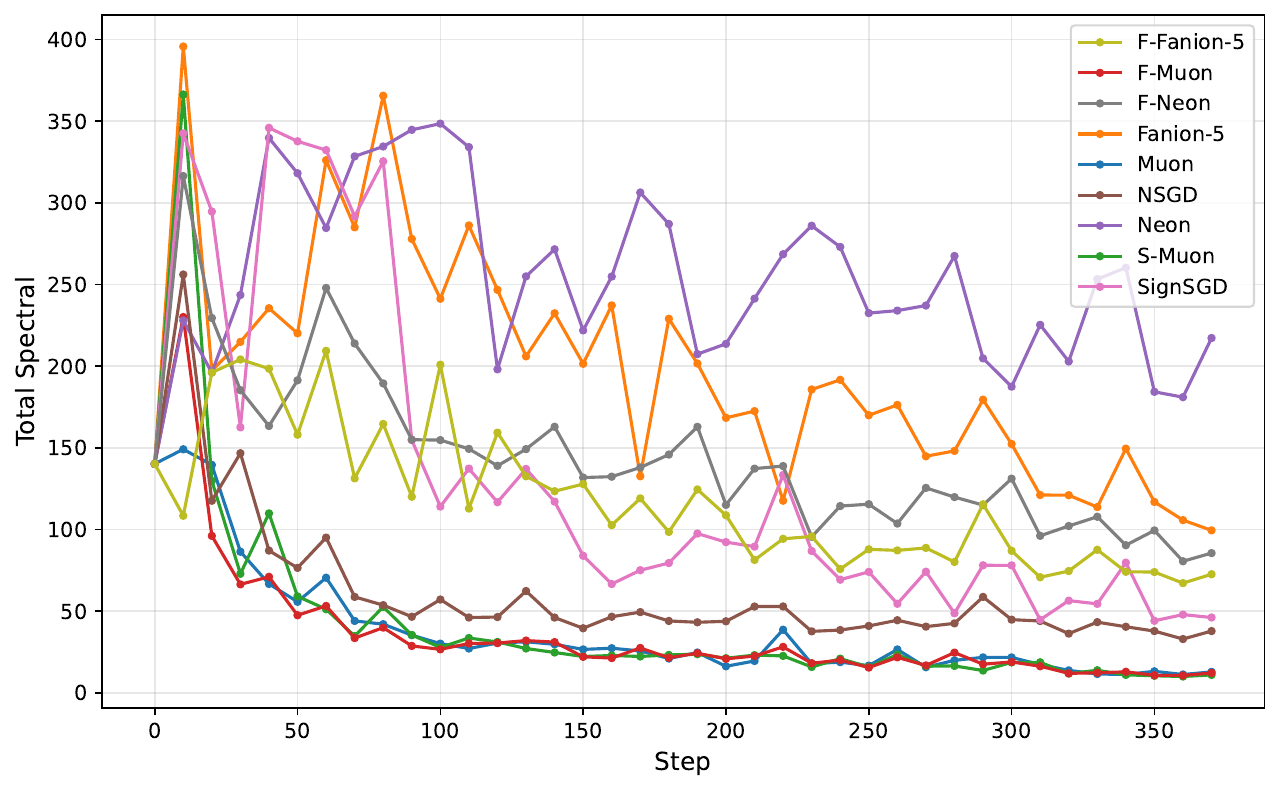}}%
       \vspace{1em}
       \subfigure[Total nuclear norm]{\label{fig:total_nuclear_norm_cifar}%
         \includegraphics[width=0.7\linewidth]{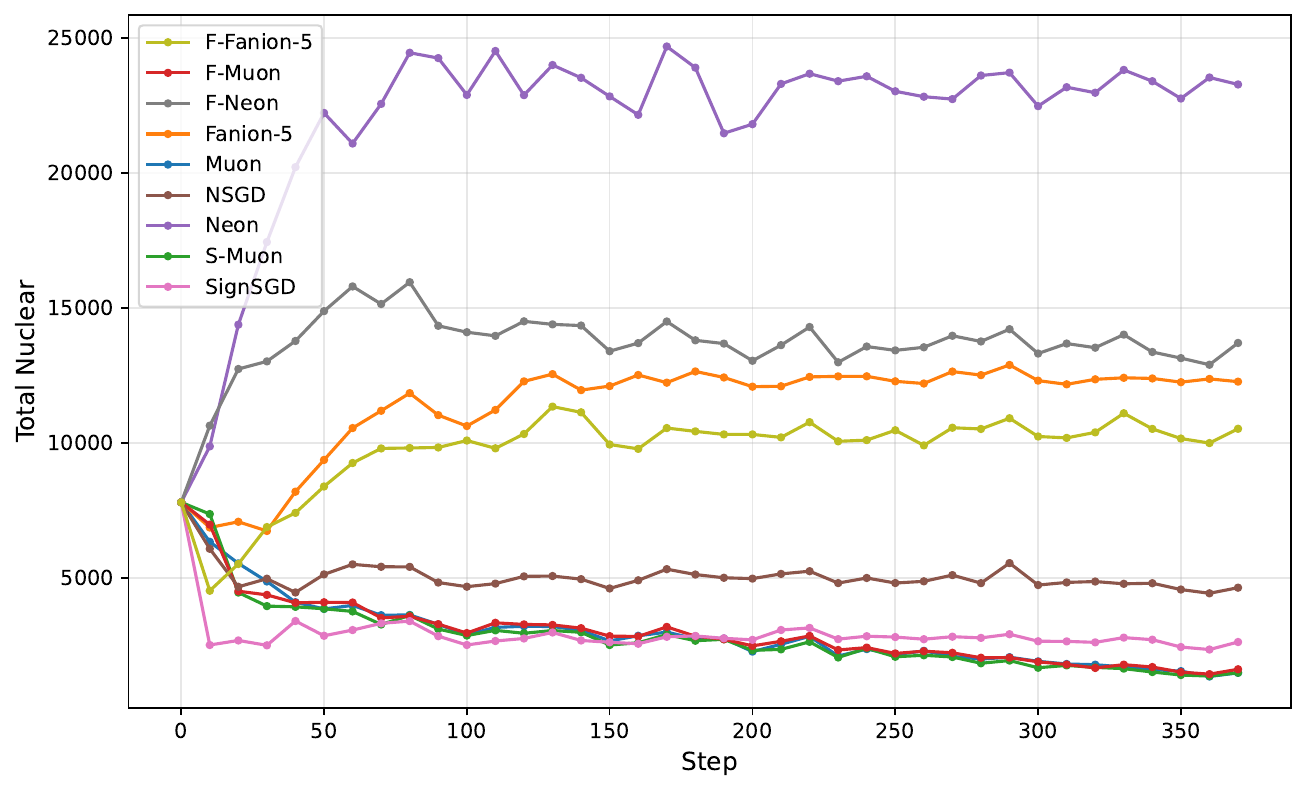}}
     
\caption{Matrix norms of the whole gradient of CIFAR CNN.}\label{fig:total_norms_cifar}

   \end{figure}
 
 \section{Technical Details of the Experiments}
 
 \paragraph{The smooth convex problem and CIFAR-10 airbench} Experiments were conducted on a single NVIDIA RTX A4000. The code was run with Python 3.10.13 and Pytorch~2.6.0+cu126 on a computer with an EPYC 7543 processor.
 
 \paragraph{NanoGPT and GPT-2 Medium Pretrain} We used the standard setting of 8 $\times$ NVIDIA H100 as documented in \citep{modded_nanogpt_2024}. The code was run with Python 3.10.18, Pytorch~2.10.0.dev20251124+cu126 on a computer with Xeon® Platinum 8462Y+ processor and 130 GB of RAM.
 
 \paragraph{NanoGPT Fine-tuning.} Experiments were conducted on a single NVIDIA RTX 4090 (24GB) GPU using PyTorch~2.8 with CUDA~12.8 and cuDNN~9.0, running on a workstation equipped with an Intel Core i9-14900KS CPU and 128 GB of RAM. No distributed or mixed-hardware training setups were used, ensuring a strictly controlled and unbiased comparison across optimizers.
 
 A uniform training protocol was applied to all runs, including an identical batch size, warm-up and cosine decay learning rate schedule, gradient clipping strategy, and validation split. Our optimization methods were integrated into the NanoGPT training loop without modifying the model architecture or preprocessing pipeline. Model quality was primarily evaluated using validation loss tracked over 200 training steps.
 
 \end{document}

%% file: math_commands.tex
\usepackage{amsmath,amsfonts,bm}

\def\eqref#1{equation~\ref{#1}}

\def\1{\bm{1}}

\def\vu{{\bm{u}}}
\def\vv{{\bm{v}}}

\def\mB{{\bm{B}}}

\def\mD{{\bm{D}}}

\def\mG{{\bm{G}}}

\def\mM{{\bm{M}}}
\def\mN{{\bm{N}}}

\def\mU{{\bm{U}}}
\def\mV{{\bm{V}}}
\def\mW{{\bm{W}}}
\def\mX{{\bm{X}}}
\def\mY{{\bm{Y}}}
\def\mZ{{\bm{Z}}}

\def\mSigma{{\bm{\Sigma}}}

\DeclareMathAlphabet{\mathsfit}{\encodingdefault}{\sfdefault}{m}{sl}
\SetMathAlphabet{\mathsfit}{bold}{\encodingdefault}{\sfdefault}{bx}{n}

\newcommand{\R}{\mathbb{R}}

\DeclareMathOperator*{\argmax}{arg\,max}
\DeclareMathOperator*{\argmin}{arg\,min}

\DeclareMathOperator{\sign}{sign}

\newcommand{\cB}{\mathcal{B}}

\newcommand{\cN}{\mathcal{N}}
\newcommand{\cO}{\mathcal{O}}
\newcommand{\cS}{\mathcal{S}}